\documentclass[12pt]{amsart}
\usepackage{amsthm, amssymb, amsmath,enumitem}
\usepackage{amsfonts, amscd, color}
\usepackage{epsfig,multicol}
\usepackage{tikz-cd}
\usepackage[colorlinks=true, linkcolor=blue, citecolor=blue, urlcolor=blue]{hyperref}

\theoremstyle{plain} \numberwithin{equation}{section}
\newtheorem{theo}{Theorem}[section]
\newtheorem{coro}[theo]{Corollary}
\newtheorem{prop}[theo]{Proposition}
\newtheorem{lemm}[theo]{Lemma}

\theoremstyle{definition}
\newtheorem{defi}{Definition}[section]

\def\C{\mathbb C}
\def\CS{\mathbb C^*}

\def\CZN{\C^n\times\Z^n}
\def\N{\mathbb N}
\def\R{\mathbb R}
\def\Q{\mathbb Q}
\def\T{\mathbb T}
\def\Z{\mathbb Z}
\def\b{b}
\def\v{v}
\def\i{\sqrt{-1}}
\def\OL{\overline}

\def\ME{\mathcal E}
\def\MF{\mathcal F}

\def\MO{\mathcal O}
\def\MR{\mathcal R}

\def\bi{\beta_i}

\def\De{\Delta}
\def\La{\Lambda}
\def\la{\lambda}

\def\RP{\R_{>0}}
\def\Si{\Sigma}

\def\vph{\varphi}
\def\bp{\beta_{I}^{\perp}}

\def\1{\mathbf 1}
\def\0{\mathbf 0}

\DeclareMathOperator{\Lie}{Lie}
\DeclareMathOperator{\Ker}{Ker}
\DeclareMathOperator{\Hom}{Hom}
\DeclareMathOperator{\Aut}{Aut}
\DeclareMathOperator{\TKVB}{TKVB}
\DeclareMathOperator{\TKVS}{TKVS}
\DeclareMathOperator{\SKVB}{SKVB}
\DeclareMathOperator{\SKVS}{SKVS}

\begin{document}

\title{Klyachko vector bundles over topological toric manifolds }
\author[Y.Cui]{Yong Cui} 
\address{Department of Mathematics, University of Maryland, Maryland, MD, USA}
\email{cuiyong@umd.edu}

\begin{abstract}
We give a Klyachko-type classification of topological/smooth/holomorphic $\T=(\CS)^n$-equivariant vector bundles that are equivariantly trivial over invariant affine charts. This generalizes Klyachko's classification of toric vector bundles.

\end{abstract}

\maketitle

\tableofcontents
\thispagestyle{empty}

\section{Introduction}
Let $X=X(\De)$ be an $n$-dimensional nonsingular toric variety over $\C$ associated with a toric fan $\De$. A toric vector bundle on  $X$ is a $\T=(\CS)^{n}$-equivariant algebraic vector bundle $\ME\to X$. In \cite{Kly90}, Klyachko classified all toric vector bundles on $X$ in terms of finite dimensional complex vector spaces $E$ with collections of $\Z$-graded filtrations $\{ E^{\alpha}(i)\}$ indexed by the rays $\alpha$ in the fan $\Delta$. A key theorem in that paper states that a toric vector bundle is equivariantly trivial over each affine piece $U_{\sigma}$ of $X$ that corresponds to a cone $\sigma\in \Delta$.\\
\indent Recently in \cite{IFM12}, the authors constructed topological toric manifolds $X(\Delta)$ from the so-called topological fans $\Delta=(\Sigma,\beta)$ that generalize the class of compact non-singular toric varieties. Similar to toric varieties, a topological toric manifold has an open dense orbit and there is a one-to-one correspondence (Theorem \ref{orbit-cone}) between the set of orbits and the set of abstract cones. We give a Klyachko type classification of a special class of so-called topological/smooth/holomorphic Klyachko vector bundles which are  topological/smooth(i.e. $C^{\infty}$)/holomorphic equivariant vector bundles over topological toric manifolds which are equivariantly trivial over each affine piece corresponding to an abstract cone. The category of topological (resp.~smooth) Klyachko vector bundles over a topological manifold $X$ is denoted by $\TKVB_{X}$ (resp.~$\SKVB_{X}$). \\
 Before stating the main theorem, we provide the following table of notations in topological toric manifolds reminiscent of notations in toric geometry. The group $\MR:=\C\times \Z$ of continuous/smooth (co)characters of $\CS$ is equipped with a ring structure and a bracket operation $\langle, \rangle$. We define two partial orders $\geq_{c},\geq_{s}$ on $\MR$ where $c$ stands for continuous and $s$ stands for smooth. 

\begin{table}[h]
    \centering
    \begin{tabular}{|c|c|c|}
        \hline
         &Toric Varieties &TTM \\ \hline
         Fans &toric fan $\De$ &topological fan $\De=(\Si,\beta)$ (Def.~\ref{top-fan})\\ \hline
        Cones &$\sigma$ &$I$ (Def.~\ref{simcom}) \\ \hline
        Affine charts &$U_{\sigma}$ & $U_{I}$(Def.~\ref{ttm}) \\ \hline
        character group of $\T$ &$M$ &$\MR^{n}$(Lemma~\ref{allcharacters}) \\ \hline
        cocharacter group of $\T$ &$N$ &$\MR^{n}$(Lemma~\ref{allcharacters}) \\ \hline
        orders & total order $\geq$ on $\R$ &partial orders $\geq_{c},\geq_{s}$ on $\MR$(Def.~\ref{geqc},\ref{geqs})\\ \hline
        distinguished points &$\gamma_{\sigma}$ &$\gamma_{I}$(Def.~\ref{orbit})\\ \hline
        stabilizer subgroup of $\gamma_{\sigma}$ and $\gamma_{I}$ &$T_{\sigma}$ &$\T_{I}$(Def.~\ref{bp}) \\ \hline
        character groups of $T_{\sigma}$ and $\T_{I}$ &$M_{\sigma}=M/(\sigma^{\perp}\bigcap M)$ &$\MR^{n}/\bp$(Def.~\ref{bp}, Lemma~\ref{character}) \\ \hline
    \end{tabular}
    \caption{Comparison of notations}
    \label{tab:example}
\end{table}

To motivate our main theorems, we recall Klyachko's classical result summarized in \cite{Pay07}. 
\begin{defi}[Klyachko's Compatibility Condition, cf.~p2 of \cite{Pay07}]
    Let $\De$ be a toric fan. For each cone $\sigma\in \De$, there is a decomposition $E=\bigoplus_{[u]\in M_{\sigma}}E_{[u]}$ such that 
    \[E^{\rho}(i)=\sum_{[u](v_{\rho})\geq i}E_{[u]} \]
    for every $\rho\preceq \sigma$ and $i\in \Z$.
\end{defi}

\begin{theo} [cf.~p5 of \cite{Pay07}]
    The category of toric vector bundles on $X(\De)$ is naturally equivalent to the category of finite-dimensional $k$-vector spaces $E$ with collections of decreasing filtrations $\{E^{\rho}(i)\}$ indexed by the rays of $\De$, satisfying Klyachko's Compability Condition. 
\end{theo}

\begin{defi} [Def.~\ref{TKVS}]
For a topological toric manifold $X$, the category $\TKVS_{X}$ is the category whose objects are vector spaces $E$ with a poset $P$ of subspaces $E^{i}(\mu)$ of $E$ indexed by $\mu\in \MR,i\in \Si^{(1)}$ which satisfy the following compatibility condition:\\
    For any $I\in \Si$, there exists a $\MR^{n}/\bp$-grading   
    \[E=\bigoplus_{[\chi]\in \MR^{n}/\bp}E_{[\chi]}\] for which
    \[E^{i}(\mu)=\sum_{\langle \chi,\beta_{i} \rangle {\geq_{c}}\mu}E_{[\chi]},\hspace{0.2cm} \forall i\in I.\]
    A morphism $f:(E,\{E^{i}(\mu)\})\to (F,\{F^{i}(\mu)\})$ in $\TKVS_{X}$ is a linear transformation $f:E\to F$ such that $f(E^{i}(\mu))\subset F^{i}(\mu)$ for all $i\in \Si^{(1)}$ and $\mu\in \MR$. $f$ is an isomorphism if $E\to F$ is an isomorphism of complex vector spaces. $f$ is a monomorphism if $E\to F$ is a monomorphism of complex vector spaces. $f$ is an epimorphism if $E\to F$ is an epimorphism of complex vector spaces and $f(E^{i}(\mu))=F^{i}(\mu)$ for all $i\in \Si^{(1)}$ and $\mu\in \MR$.
\end{defi}
The category $\SKVS$ is defined similarly with $\geq_{c}$ replaced by $\geq_{s}$.

The main theorems states that 
\begin{theo}[Theorems \ref{main}, \ref{main2}]
    $\TKVB_{X}\simeq \TKVS_{X}, \SKVB_{X}\simeq \SKVS_{X}.$
\end{theo}

The holomorphic Klyachko vector bundles degenerate to the toric case: 
\begin{theo} [Theorem \ref{main4}]
    A holomorphic Klyachko vector bundle $\ME$ over a complex topological toric manifold $X=X(\De)$ is $\T$-equivariantly biholomorphic to a toric vector bundle over a toric variety. 
\end{theo}

To deal with a canonical exact sequence (Theorem \ref{canonical}) of equivariant complex vector bundles where morphisms are $\R$-linear in fibers, we introduce the following two categories:
\begin{defi} [Def. \ref{SKVB-real}]
    For a topological toric manifold $X$, the category $\widetilde{\SKVB_{X}}$ is the category whose objects are smooth Klyachko vector bundles over $X$ and whose morphisms are morphisms of equivariant complex vector bundles that are $\R$-linear in fibers.
\end{defi}

\begin{defi} [Def. \ref{SKVS-real}]
For a topological toric manifold $X$, the category $\widetilde{\SKVS_{X}}$ is the category having the same objects as $\SKVS_{X}$. 
    A morphism $f:(E,\{E^{i}(\mu)\})\to (F,\{F^{i}(\mu)\})$ in $\widetilde{\SKVS_{X}}$ is an $\R$-linear map $f:E\to F$ such that $f(E^{i}(\mu))\subset F^{i}(\mu)$ for all $i\in \Si^{(1)}$ and $\mu\in \MR$. $f$ is an isomorphism if $E\to F$ is an isomorphism of real vector spaces. $f$ is a monomorphism if $E\to F$ is a monomorphism of real vector spaces. $f$ is an epimorphism if $E\to F$ is an epimorphism of real vector spaces and $f(E^{i}(\mu))=F^{i}(\mu)$ for all $i\in \Si^{(1)}$ and $\mu\in \MR$.
\end{defi}

The fourth main theorem states that:
\begin{theo} [Theorem \ref{main3}]
    The category $\widetilde{\SKVB_{X}}$ is equivalent to the category $\widetilde{\SKVS_{X}}$.
\end{theo}

In the companion paper \cite{CG} and future work, we plan to answer the following questions:
\begin{itemize}
    \item Is every topological (resp.~smooth) $\T$-equivariant vector bundle a topological (resp.~smooth) Klyachko vector bundle?
    \item How to compute cohomology groups/rings of topological toric manifolds with coefficients in Klyachko vector bundles using their poset data?
    \item How to compute Chern numbers of Klyachko vector bundles using their poset data?
\end{itemize}

\indent The paper is organized as follows: In §1, we review the construction of topological toric manifolds. In §2, we prove the orbit-cone correspondence for topological toric manifolds that generalizes the toric case. In §3,4, we prove the main theorems \ref{main}, \ref{main2} and \ref{main3}. In §5, we show that a short exact sequence of $(\CS)^{n}$-equivariant vector bundles constructed in theorem 6.5 of \cite{IFM12} is a short exact sequence of smooth Klyachko vector bundles and give another proof of the exactness of this sequence by theorem~\ref{main3}. In §6 we show that a holomorphic Klyachko vector bundle over a topological toric manifold with $(\CS)^{n}$-equivariant complex structure is exactly a toric vector bundle over a nonsingular complete toric variety. \\
$ $\\
\noindent
{\bf Notation.}
Throughout, we denote by $\T$ the complex algebraic torus $(\CS)^{n}$ where $n$ is the dimension of the topological toric manifold considered. 

\bigskip

\noindent
{\bf Acknowledgment.}  
I am grateful to Professor Tristan Hübsch for introducing me to the topic of torus manifolds in the study of mirror symmetry and to my advisor Professor Amin Gholampour for many helpful discussions when writing this paper and to my second advisor Professor Patrick Brosnan for proofreading my earlier draft. I would also like to thank Professor Mikiya Masuda and Professor Hiroaki Ishida for answering many questions about torus manifolds. 
\bigskip

\section{Background on topological toric manifolds}
Topological toric manifolds are generalizations of smooth toric varieties. Like toric varieties, they are determined by topological toric fans. Here we review the construction in \cite{IFM12}.

\begin{defi}
 Let $S$ be the ring of $2\times 2$ matrices of the form $\begin{bmatrix}
     b &0\\
     c &v
 \end{bmatrix}$ where $b,c\in \R,v\in \Z$.
Denote the set $\C\times \Z$  by $\MR$. We define a ring structure on $\MR$ via the bijection $\MR\to S$ sending $\mu=(b+\i c,v)$ to $\begin{bmatrix}
     b &0\\
     c &v
 \end{bmatrix}$, i.e. 
 \begin{align*}
  \mu_1+\mu_2&:=((b_1+b_2)+\sqrt{-1}(c_1+c_2),v_1+v_2)\\  
  \mu_1\mu_2&:=(b_1b_2+\sqrt{-1}(b_1c_2+c_1v_2),v_1v_2).
 \end{align*}
 for $\mu_i=(b_i+\sqrt{-1}c_i,v_i),i=1,2.$\\
 For $g\in \CS, \mu=(b+\i c,v)\in \MR$, define 
 \[
 g^\mu:=|g|^{b+\i c}(\frac{g}{|g|})^v,
 \]
then $(g^{\mu_1})^{\mu_2}=g^{\mu_2\mu_1}$.\\
The identity element $\1=(1,1)$ of $\MR$ corresponds to $\begin{pmatrix}
    1 &0\\
    0 &1
\end{pmatrix}$.
The ring $\MR^{n}$ plays the roles of the character group and the cocharacter group for topological toric manifolds.
\end{defi}
 
\begin{defi}
Let $G$ be an abelian topological group. A continuous (resp.\hspace{0.1cm}smooth) character of $G$ is a continuous (resp.\hspace{0.1cm}smooth) homomorphism $G\to \CS$. A continuous (resp.\hspace{0.1cm}smooth) cocharacter of $G$ is a continuous (resp.\hspace{0.1cm}smooth) homomorphism $\CS\to G$. 
\end{defi}
 
 For $\alpha=(\alpha^1,\dots,\alpha^n)\in \MR^n$ and $\beta=(\beta^1,\dots,\beta^n)\in \MR^n$, define smooth characters $\chi^\alpha\in \Hom((\CS)^n,\CS)$ and smooth cocharacters $\la_\beta\in \Hom(\CS,(\CS)^n)$ via
 \[
 \chi^\alpha(g_1,\dots,g_n):=\prod_{k=1}^n g_k^{\alpha^k},  \quad\la_\beta(g):=(g^{\beta^1},\dots,g^{\beta^n}).
 \]

 \begin{lemm} \label{allcharacters}
 All continuous characters (resp.~cocharacters) of $\T$ are of the form $\chi^{\alpha}$ (resp.~$\la_{\beta}$). In particular, every continuous character of $\T$ is smooth.    
 \end{lemm}

\begin{proof}
Given a continuous character $\chi:\CS\to\CS$, because $\CS\simeq \RP\times S^{1}$ as an abelian topological group, the restriction $\chi_{\RP}$ of $\chi$ to $\RP$ and the restriction $\chi_{S^{1}}$ of $\chi$ to $S^{1}$ are also continuous homomorphisms. $\RP$ is isomorphic to $(\R,+)$ as an abelian topological group via $\log$, therefore $\chi_{\RP}(r)=r^{b+ci}$ for some $b,c\in \R$. Let $\pi:\R\to S^{1}$ be the covering map $\theta\mapsto e^{2\pi i\theta}$. Let $\phi:\R\to \R$ be the unique lifting of $\chi_{S^{1}}\circ \pi$ with $\phi(0)=0$ so that the following diagram is commutative:
\begin{center}
    \begin{tikzcd}
    \R \ar[r,"\phi"] \ar[d,"\pi"']&\R \ar[d,"\pi"]\\
    S^{1} \ar[r,"\chi_{S^{1}}"] &S^{1}
    \end{tikzcd}
\end{center}
We prove that $\phi$ is a homomorphism. Let $\theta_{1},\theta_{2}\in \R$. We have $$\chi_{S^{1}}(\pi(\theta_{1}+\theta_{2}))=\chi_{S^{1}}(\pi(\theta_{1}))\chi_{S^{1}}(\pi(\theta_{2})).$$ By commutativity of the diagram, this is equivalent to $$\pi(\phi(\theta_{1}+\theta_{2}))=\pi(\phi(\theta_{1})+\phi(\theta_{2})).$$ So $f(\theta_{1},\theta_{2}):=\phi(\theta_{1}+\theta_{2})-\phi(\theta_{1})-\phi(\theta_{2})\in \Z$. Since $f$ is continuous, $f$ must be constant. This constant must be $0$ because $\phi(0)=0$. Hence $\phi$ is a homomorphism. Let $v=\phi(1)$. It's easy to see that $\phi(n)=nv$ for $n\in \Q$. By continuity of $\phi$, we know $\phi(r)=rv$ for all $r\in \R$. By continuity of the above diagram, $v\in \Z$. Hence $\chi(z)=z^{\mu}$ where $\mu=(b+ci,v)\in \MR$. \\
In general, given a continuous character $\chi:(\CS)^{n}\to \CS$, its restriction $\chi_{i}$ to the $i$-th coordinate subspace of $(\CS)^{n}$ is $\chi^{\alpha^{i}}$ for some $\alpha^{i}\in \MR$. Hence $\chi=\chi^{\alpha}$ where $\alpha=(\alpha^1,\cdots,\alpha^n)$. \\
Notice that a continuous cocharacter of $\CS$ is the same as a continuous character of $\CS$, so we can prove the claim about continuous cocharacters similarly. 
\end{proof}

\begin{defi}
Define a bracket operation by 
 \[
 \langle \alpha,\beta\rangle:=\sum_{k=1}^n \alpha^k\beta^k\in \MR.
 \]
 Notice that in general $\langle \alpha,\beta\rangle\neq \langle \beta,\alpha\rangle$ because the matrix product is not commutative.\\ 
\indent For a family $\{\alpha_{i}\}_{i=1}^{n}$ of $n$ elements in $\MR^{n}$, define its dual to be a family $\{\beta_i\}_{i=1}^{n}$ such that 
$\langle \alpha_{i},\beta_{j}\rangle=\delta_{ij}\1$ for all $i,j$, which is equivalent to  
\begin{equation} \label{dual}
\begin{bmatrix} \alpha^1_1 & \alpha^2_1 & \dots &\alpha^n_1 \\
\alpha^1_2 & \alpha^2_2 & \dots &\alpha^n_2 \\
\vdots &\vdots & \ddots & \vdots\\
\alpha^1_n & \alpha^2_n & \dots &\alpha^n_n 
\end{bmatrix} 
\begin{bmatrix} \beta^1_1& \beta^1_2 & \dots &\beta^1_n\\
\beta^2_1& \beta^2_2 & \dots &\beta^2_n\\
\vdots &\vdots & \ddots & \vdots\\
\beta^n_1& \beta^n_2 & \dots &\beta^n_n
\end{bmatrix} =\begin{bmatrix} \1 & \0 &\dots &\0\\
\0&\1&\dots &\0\\
\vdots &\vdots & \ddots & \vdots\\
\0&\0&\dots &\1\end{bmatrix}
\end{equation} where $\alpha_j=(\alpha_j^1,\dots,\alpha_j^n),\beta_{j}=(\beta_j^1,\dots,\beta_j^n)^{T}$.\\
\indent For $\{\alpha_i\}_{i=1}^{n}$ and $\{\beta_i\}_{i=1}^{n}$ with $\alpha_i,\beta_i\in \MR^n$, we define group endomorphisms $\bigoplus_{i=1}^{n}\chi^{\alpha_{i}}$ and $\prod_{i=1}^{n}\la_{\beta_{i}}$ of $(\CS)^{n}$ by 
\begin{equation} 
\begin{split}
&(\bigoplus_{i=1}^n\chi^{\alpha_i})(g_1,\dots,g_n):=(\chi^{\alpha_1}(g_1,\dots,g_n),\dots, \chi^{\alpha_n}(g_1,\dots,g_n))\\
&(\prod_{i=1}^n\la_{\beta_i})(g_1,\dots,g_n):=\prod_{i=1}^n\la_{\beta_i}(g_i).
\end{split}
\end{equation}    
\end{defi}

\begin{lemm} [cf.~Lemma 2.4 of \cite{IFM12} ]
If $\{b_i\}_{i=1}^n$ and $\{v_i\}_{i=1}^n$ are bases of $\R^n$ and $\Z^n$ respectively, then $\{\beta_i\}_{i=1}^n$ has a dual set $\{\alpha_i\}_{i=1}^n$, i.e. $\langle \alpha_i,\beta_j\rangle=\delta_{ij}\1$ for any $i,j$.   
\end{lemm}

\begin{lemm}[cf.~Lemma 2.3 of \cite{IFM12}] \label{lem1}
 If $\{\alpha_i\}_{i=1}^n$ is dual to $\{\beta_i\}_{i=1}^n$, then the composition 
$$\big(\prod_{i=1}^n \la_{\beta_i}\big)\big(\bigoplus_{i=1}^n\chi^{\alpha_i}\big)\colon (\C^*)^n\to (\C^*)^n$$ 
is the identity,  in particular, both $\bigoplus_{i=1}^n \chi^{\alpha_i}$ and $\prod_{i=1}^n \la_{\beta_i}$ are automorphisms of $(\C^*)^n$.     
\end{lemm}

\begin{defi} \label{simcom}
An abstract simplicial complex $\Si$ is a collection of non-empty finite subsets of a set $S$ such that for every set $I\in \Si$ and $J\subset I$, we have $J\in \Si$. Let $\Si$ be an abstract simplicial complex. If $\Si$ contains only finitely many elements, then it is called a finite abstract simplicial complex. The dimension of $I\in \Si$ is $|I|-1$. If $|I|\leq n$ for all $I\in \Si$ and there exists $I_{0}\in \Si$ such that $|I_{0}|=n$, then $\Si$ is called an abstract simplicial complex of dimension $n-1$. An augmented abstract finite simplicial complex $\Si$ of dimension $n-1$ is a set $\Si'\bigcup \{\emptyset\}$ where $\Si'$ is an abstract finite simplicial complex of dimension $n-1$.
For $i\in \N$, denote by $\Si^{(i)}$ the set of elements $I\in \Si$ such that $|I|=i$. In particular, we denote the vertex set by $\Si^{(1)}$.
\end{defi}

\begin{defi} \label{top-fan}
Let $\Si$ be an augmented abstract finite simplicial complex of dimension $n-1$ and let 
\[
\beta\colon \Si^{(1)} \to \MR^n=\CZN
\]
where $\Si^{(1)}$ denotes the vertex set of $\Si$.  We abbreviate an element $\{i\}\in\Si^{(1)}$ as $i$ and $\beta(\{i\})$ as $\beta_i$ and express $\beta_i=(b_i+\sqrt{-1}c_i,v_i)\in\CZN$. Denote the cone spanned by $b_{i},i\in I$ by $\angle b_{I}$. Then the pair $(\Si,\beta)$ is called a \emph{(simplicial) topological fan} of dimension $n$ if the following are satisfied.
\begin{enumerate}
\item  $b_i$'s for $i\in I$ are linearly independent whenever $I\in \Si$, and $\angle \b_I\cap \angle \b_J=\angle \b_{I\cap J}$ for any $I,J\in \Si$.
(In short, the collection of cones $\angle b_I$ for all $I\in \Si$ is an ordinary simplicial fan although $\b_i$'s are not necessarily in $\Z^n$.) 
\item Each $\v_i$ is primitive and $v_i$'s for $i\in I$ are linearly independent (over $\R$) whenever $I\in \Si$.  
\end{enumerate}     
A topological fan $\De$ of dimension $n$ is \emph{complete} if $\bigcup_{I\in \Si}\angle \b_I=\R^n$ and \emph{non-singular} if the $\v_i$'s for $i\in I$ form a part of a $\Z$-basis of $\Z^n$ whenever $I\in\Si$.  
\end{defi}

\noindent 
{\bf Construction.}
Let $\De=(\Si,\beta)$ be a non-singular topological fan of dimension $n$. We take the vertex set $\Si^{(1)}$ as $[m]=\{1,2,\dots,m\}$. Because $\Si$ is a fan, for all $I\in \Si$, we have $I\subset [m]$. For $I\subset [m]$, we set $I^{c}:=\Si^{(1)}\backslash I$ and \[U(I):=\{(z_1,\dots,z_m)\in \C^m\mid z_i\not=0 \ \text{for $\forall i\notin I$}\}=\C^{I}\times (\CS)^{I^{c}}.\]
Note that $U(I)\cap U(J)=U(I\cap J)$ for any $I,J\in [m]$ and $U(I)\subset U(J)$ if and only if $I\subset J$.  We define 
\[
U(\Sigma):=\bigcup_{I\in \Si}U(I).
\]
Let 
\[
\la:(\C^*)^m\to (\C^*)^n
\]
be the homomorphism defined by
\[
\la(h_1,\dots,h_m):=\prod_{k=1}^m\la_{\beta_k}(h_k).
\]

\begin{lemm}
    $\la$ is surjective.
\end{lemm}

\begin{proof}
Given $I\in \Si^{(n)}$, let $\{\alpha^I_i\}_{i\in I}$ be the dual to $\{\beta_i \}_{i\in I}$. We prove that $\bigoplus_{i\in I} \chi^{\alpha^I_i}\circ \la$ is surjective. By Lemma~\ref{lem1} $\bigoplus_{i\in I} \chi^{\alpha^I_i}$ is a bijection, hence $\la$ is surjective.\\
Surjectivity of $\bigoplus_{i\in I} \chi^{\alpha^I_i}\circ \la$: Given $(z_1,\dots,z_n)\in (\CS)^n=(\CS)^I$, let $h_k=z_k$ if $k\in I$ and $h_k=1$ otherwise, then 
\[\chi^{\alpha^I_i}\circ \la(h_1,\dots,h_m)=\prod_{k=1}^mh_k^{\langle \alpha,\beta_k\rangle}=z_i.\]
Hence $\bigoplus_{i\in I} \chi^{\alpha^I_i}\circ \la(h_1,\dots,h_m)=(z_1,\dots,z_n)$.
\end{proof}

\begin{lemm}[cf.~Lemma 4.1 of \cite{IFM12}] \label{lemm:kerl} 
\begin{equation} \label{kerl}
\Ker\la=\{(h_1,\dots,h_m)\in (\C^*)^m
\mid \prod_{k=1}^mh_k^{\langle \alpha,\beta_k\rangle}=1\quad \text{for any $\alpha\in \MR^n$}\}.
\end{equation}
Using $\{\alpha^I_i\}_{i\in I}$ for $I\in \Si^{(n)}$, which is dual to $\{\beta_i\}_{i\in I}$, we have  
\begin{equation} \label{aibk}
\Ker\la=\{(h_1,\dots,h_m)\in (\C^*)^m \mid h_i\prod_{k\notin I}h_k^{\langle \alpha^I_i,\beta_k\rangle}=1\quad \text{for any $i\in I$}\}.
\end{equation}
\end{lemm}

\begin{defi} \label{ttm}
 We call 
\[X(\De):= U(\Si)/\Ker \la \] equipped with the quotient topology together with the natural action of $\T=(\CS)^n\simeq (\CS)^m/\Ker \la$
the topological toric manifold associated with $\De$. \\
Denote $U(I)/\Ker\la$ by $U_{I}$ and the representation space of $\bigoplus_{i\in I}\chi^{\alpha^{I}_{i}}$ by $\bigoplus_{i\in I}V(\chi^{\alpha^{I}_{i}})\simeq \C^{I}$. The following theorem confirms that $X(\De)$ is a smooth manifold.   
\end{defi}

\begin{theo}[cf.~Lemma 5.1 and Lemma 5.2 of \cite{IFM12}] \label{chart}
   $X(\De)$ is a smooth manifold with a smooth $\T$-action whose equivariant local charts are the maps \[\varphi_{I}:U_{I}\to \bigoplus_{i\in I}V(\chi^{\alpha^{I}_{i}}) \]
defined via 
\[\varphi_{I}(z_1,\dots,z_m):= (\prod_{k=1}^{m}z_{k}^{\langle \alpha^{I}_{i},\beta_{k}\rangle})_{i\in I} \] for all $I\in \Si^{(n)}$ and whose $j$-th component of transition function $\varphi_{J}\varphi_{I}^{-1}:\varphi_{I}(U_{I}\bigcap U_{J})\to \varphi_{J}(U_{J})$ is given by 
\[\prod_{i\in I}w_{i}^{\langle \alpha^{J}_{j},\beta_{i}\rangle}. \] for all $I,J\in \Si^{(n)}$.
\end{theo}

\noindent {\bf Remark}: If $\De$ is such that $\beta_{i}=(v_{i},v_{i})\in \Z^{n}\times \Z^{n}$ for each $i\in \Si^{(1)}$, then $X(\De)$ is just the toric variety associated with the fan $\Si'$ whose cones are all $\sigma_{I}=\angle(\{v_{i}|i\in I\})$.

\bigskip


\section{Orbit-Cone correspondence}
\begin{defi} \label{orbit}
    Let $\De=(\Si,\beta)$ be a non-singular topological fan of dimension $n$.
For each $I\in \Si$, we call $$\gamma_{I}=[0^{I}\times 1^{\Si^{(1)}\backslash I}]\in U(I)/\Ker\la$$ the distinguished point corresponding to $I$ and we denote its $\T$-orbit by $$\mathcal{O}(I):=\T \cdot \gamma_{I}=(0^{I}\times (\C^{*})^{\Si^{(1)}\backslash I})/\Ker\la$$
and the closure of $\MO(I)$ in $X(\De)$ by 
$$V(I):=\overline{\MO(I)}=(0^{I}\times \C^{\Si^{(1)}\backslash I})/\Ker\la\subset X(\De).$$
\end{defi}
The following theorem is a direct generalization of Theorem 3.2.6 of \cite{CLS}.
\begin{theo} \label{orbit-cone}
$ $\\
    \begin{enumerate}
        \item There is a bijective correspondence 
            \begin{align*}
            \Si &\longleftrightarrow \{\T\text{-orbits in } X(\De)\}\\
            I &\mapsto \MO(I)
            \end{align*}
        \item For each $J\in \Si^{(j)}, dim_{\C}\MO(J)=n-j$.
        \item \[U_{I}=\bigcup_{J\leq I}\MO(J).\]
        \item $J\leq I$ iff $\MO(I)\subset V(J)$ and \[V(J)=\bigcup_{J\leq I}\MO(I).\]
    \end{enumerate}
\end{theo}
\begin{proof}
$ $\\
    \begin{enumerate}
        \item Let $\mathcal{O}$ be a $\T$-orbit in $X(\Delta)$ and $\gamma=[\gamma^{1},\dots,\gamma^{m}]\in \MO$. Let $I=\{i\in \Sigma(1)| \gamma^{i}=0\},$ then $\MO\bigcap \MO(I)\neq \emptyset$. Hence $\MO=\MO(I)$.
        \item Given $J\in \Sigma^{(j)}$, let $I\in \Sigma^{(n)}$ such that $J\leq I$. Then $\vph_{I}(\MO(J))=0^{J}\times (\C^{*})^{I\backslash J}$, hence $dim_{\C}\MO(J)=n-j$.
        \item follows from (a).
        \item If $J\leq I$, then by definition $\MO(I)\subset V(J)$. Conversely if $\MO(I)\subset V(J)$, then $\Sigma(1)\backslash I\subset \Sigma(1)\backslash J$, i.e. $J\leq I$.
    \end{enumerate}
\end{proof}

\section{Topological Klyachko vector bundles}
\begin{defi}
A topological Klyachko vector bundle of rank $r$ over an $n$-dimensional topological toric manifold $X(\De)$ with a $\T$-equivariant complex structure associated with a topological toric fan $\De=(\Si,\beta)$ is a $\T$-equivariant topological complex vector bundle $\ME\to X(\De)$ whose restriction $\left. \ME \right|_{U_{I}}$ is equivariantly homeomorphic to $U_{I}\times \C^{r}$ for all $I\in \Si$ and whose action map $\T\times \ME\to \ME$ is continuous.
\end{defi}

\begin{lemm} \label{ext1}
A continuous character $\chi^{\alpha}:\CS\to \CS$ where $\alpha=(b+c\sqrt{-1},v)\in \MR$ can be continuously extended to $\C$ iff $b>0$ or $b=c=v=0$.  
\end{lemm}

\begin{proof}
In this proof, I will use $i$ for $\sqrt{-1}$. 
Given $\chi^{\alpha}:\CS\to \CS$ with $\alpha=(b+ci,v)$ and $g=\rho e^{\theta i}\in \CS$ where $\rho\in \RP$, we have $\chi^{\alpha}(g)=\rho^{b+ci}e^{i\theta v}$. If $b>0$, then $|\chi^{\alpha}(g)|=\rho^{b}\to 0$ as $g\to 0$, therefore $\chi^{\alpha}(g)\to 0$ as $g\to 0$. If $b<0$, then $|\chi^{\alpha}(g)|=\rho^{b}\to +\infty$ as $g\to 0$, therefore $\chi^{\alpha}(g)$ diverges as $g\to 0$. If $b=0$, suppose $c\neq 0$, then let $g(t)=t, t\in (0,+\infty)$, we have $\chi^{\alpha}(g(t))=e^{ic\ln{t}}$ diverges as $t\to 0^{+}$. Therefore  $c=0$ when $\chi^{\alpha}$ can be continuously extended to $\C$ and $b=0$. When $b=c=0$, let $g(t)=\frac{e^{ti}}{t}, t\in (0,+\infty)$, then $\chi^{\alpha}(g(t))$ converges as $t\to +\infty$ if and only if $v=0$. Therefore when $b=0$, $\chi^{\alpha}$ can be continuously extended to $\C$ if and only if $c=v=0$. 
\end{proof}

\begin{defi} \label{geqc}
We can define a partial order $\geq_{c}$ on $\MR$ as follows: For $\alpha=(b+c\sqrt{-1},v)\in \MR$, we say $\alpha \geq_{c} 0=(0+0\sqrt{-1},0)$ if $b>0$ or $b=c=v=0$. For $\alpha,\beta\in \MR$, we say $\alpha\geq_{c}\beta$ if $\alpha-\beta\geq_{c} 0$.
\end{defi}

\begin{prop}
    $\geq_{c}$ is a partial order.
\end{prop}

\begin{proof}
Reflexivity is clear. Anti-symmetry: suppose $\alpha_{i}=(b_{i}+c_{i}\sqrt{-1},v_{i})\in \MR,i=1,2$ and $\alpha_{1}\geq_{c} \alpha_{2}, \alpha_{2}\geq_{c}\alpha_{1},$ then $b_{1}-b_{2}\geq 0, b_{2}-b_{1}\geq 0$, therefore $b_{1}-b_{2}=c_{1}-c_{2}=v_{1}-v_{2}=0$. Transitivity: Suppose $\alpha_{i}=(b_{i}+c_{i}\sqrt{-1},v_{i})\in \MR,i=1,2,3$ and $\alpha_{1}\geq_{c}\alpha_{2},\alpha_{2}\geq_{c}\alpha_{3}$. There are four cases: 1) $b_{1}>b_{2},b_{2}>b_{3}$, 2) $b_{1}>b_{2},\alpha_{2}=\alpha_{3}$, 3) $\alpha_{1}=\alpha_{2},b_{2}>b_{3}$, 4) $\alpha_{1}=\alpha_{2}=\alpha_{3}$. In the first three cases, we have $b_{1}>b_{3}$, hence $\alpha_{1}\geq_{c}\alpha_{3}$. The case 4 obviously satisfies $\alpha_{1}\geq_{c}\alpha_{3}$.    
\end{proof}

\begin{lemm} \label{lispan}
  If $I\in \Si^{(n)}$ and $\alpha\in \MR^{n}$, then $\alpha=\Si_{i\in I}\la_{i}\alpha^{I}_{i}$ where $\la_{i}=\langle \alpha,\bi\rangle$ and $\{\alpha^{I}_{i}\}_{i\in I}$ is dual to $\{\beta_{i}\}_{i\in I}$.
\end{lemm}

\begin{proof}
Let $\beta=\alpha-\Si_{i\in I}\la_{i}\alpha^{I}_{i}$, then $\langle \beta,\bi\rangle =0$ for all $i$.  The duality relation $\langle \alpha_{i},\beta_{j}\rangle=\delta_{ij}\1$ is equivalent to  
\begin{equation} 
\begin{bmatrix} \alpha^1_1 & \alpha^2_1 & \dots &\alpha^n_1 \\
\alpha^1_2 & \alpha^2_2 & \dots &\alpha^n_2 \\
\vdots &\vdots & \ddots & \vdots\\
\alpha^1_n & \alpha^2_n & \dots &\alpha^n_n 
\end{bmatrix} 
\begin{bmatrix} \beta^1_1& \beta^1_2 & \dots &\beta^1_n\\
\beta^2_1& \beta^2_2 & \dots &\beta^2_n\\
\vdots &\vdots & \ddots & \vdots\\
\beta^n_1& \beta^n_2 & \dots &\beta^n_n
\end{bmatrix} =\begin{bmatrix} \1 & \0 &\dots &\0\\
\0&\1&\dots &\0\\
\vdots &\vdots & \ddots & \vdots\\
\0&\0&\dots &\1\end{bmatrix}
\end{equation} where $\beta_{j}=(\beta_j^1,\dots,\beta_j^n)^{T}$, therefore the matrix $(\beta^{i}_{j})_{ij}$ is invertible as a $2n\times 2n$ matrix with coefficients in $\R$, hence $\beta=0$. 
\end{proof}

The following lemma is used to derive the compatibility condition that guarantees that we can extend the transition functions of the equivariant vector bundles from the open dense subset $\T\bigcap U_{I}\bigcap U_{J}$ to $U_{I}\bigcap U_{J}$.
\begin{lemm} \label{cext}
 If $I\in \Si$, then $\chi^{\alpha}:\T\to \CS, \alpha\in \MR^{n}$ can be continuously extended to $U_{I}$ iff $\langle \alpha,\bi\rangle \geq_{c} 0 \hspace{0.2cm}\forall i\in I$.  
\end{lemm}

\begin{proof}
We first prove the case where $I\in \Si^{(n)}$. By lemma ~\ref{lispan}, we may assume that $\alpha=\Si_{i\in I}\la_{i}\alpha^{I}_{i}$ where $\la_{i}=\langle \alpha,\bi\rangle$. The embedding $\iota: \T\hookrightarrow U_{I}$ is given by $t\mapsto (\chi^{\alpha^{I}_{i}}(t))_{i\in I}$. The map $\chi^{\alpha}\circ \iota^{-1}:\iota(\T)\to \CS$ is given by $(z_{i})_{i\in I}\mapsto \prod_{i\in I}z_{i}^{\la_{i}}$. This map can be smoothly extended to $U_{I}$ iff $\la_{i}\geq_{c} 0 $ for all $i\in I$. \\
Next we prove the general case where $J\in \Si$ but $|J|\leq n$. Choose $I\in \Si^{(n)}$ that contains $J$. Then $U_{J}=\bigcap_{i\in I\backslash J}\{z_{i}\neq 0\}\subset U_{I}$. Hence $\chi^{\alpha}$ can be smoothly extended to $U_{J}$ iff $\langle \alpha,\bi\rangle \geq_{c} 0 \hspace{0.2cm}\forall i\in J$.
\end{proof}

\begin{defi} \label{bp}
For each $I\in \Si$, let $\T_{I}=((\CS)^{I}\times \1^{\Si^{(1)}\backslash I})/\Ker \la$ be the stabilizer of $\gamma_{I}$ and $\beta_{I}^{\perp}:=\{\alpha\in \MR^{n}| \langle \alpha,\beta_{i}\rangle=0 \hspace{0.2cm} \forall i\in I\}$.    
\end{defi}

The following lemma is used in Theorem ~\ref{main} to prove that the equivariant vector bundle we get is independent of the choice of extensions of $\varphi_{I}$. It generalizes Prop 2.1.1 of \cite{Kly90}.

\begin{lemm} \label{indext}
\leavevmode
Let $X=U_{I}$ be an affine chart of a topological toric manifold corresponding to a cone $I\in \Si^{(n)}$. 
\begin{enumerate}[label=\roman*),]
    \item Define a family of subspaces of a complex $\T$-module $E$ by 
    \begin{equation*}
       E^{i}(\mu)=\bigoplus_{\langle \chi,\beta_{i}\rangle \geq_{c}\mu} E_{\chi},i\in I,  
    \end{equation*}
   where $E_{\chi}\subset E$ is the isotypical component of the continuous character $\chi\in \MR^{n}$. Similarly define $\{F^{i}(\mu)|i\in I, \mu\in \MR\}$ for another complex $\T$-module F. Then the space of equivariant homomorphisms $\Hom_{\T}(X\times E,X\times F)$ is canonically isomorphic to the space of linear operators $\varphi:E\to F$ satisfying $\varphi(E^{i}(\mu))\subset F^{i}(\mu)$ for all $i\in I$.
    \item Two topological Klyachko vector bundles $X\times E$ and $X\times F$ are isomorphic if and only if the restrictions of the continuous representations are isomorphic $\left.E\right.|_{T_{I}}\simeq \left.F\right.|_{T_{I}}$.
\end{enumerate}   
\end{lemm}

\begin{proof}
$i)$: All topological Klyachko vector bundles on $U_{I}$ decompose into a sum of line bundles by definition. Therefore we can assume that $dim E=dim F=1$. A bundle morphism $f:X\times E\to X\times F$ is a family of linear maps $\varphi_{x}:E\to F, x\in X$. The equivariance condition implies that for $\chi_{E},\chi_{F}\in \MR^{n},e\in E$, we have $\varphi_{tx}(te)=t\varphi_{x}(e)$, i.e. $\chi_{E}(t)\varphi_{tx}(e)=\chi_{F}(t)\varphi_{x}(e)$. Fix an arbitrary point $x_{0}$ in the open dense orbit and let $\varphi=\varphi_{x_{0}}$. Then $f(tx_{0},e)=(tx_{0},\chi^{-1}_{E}\chi_{F}(t)\varphi(e))$ for all $t\in \T$. Since $f$ is continuous on $X$, either $\varphi=0$ or $\chi^{-1}_{E}\chi_{F}$ can be continuously extended to $X$.  By lemma ~\ref{cext}, the latter condition is equivalent to 
    \[\langle \chi_{F},\beta_{i} \rangle \geq_{c} \langle \chi_{E},\beta_{i} \rangle, \forall i\in I. \] In both cases $\varphi$ satisfies $\varphi(E^{i}(\mu))\subset F^{i}(\mu)$.\\
    Conversely, if $\varphi:E\to F$ satisfies $\varphi(E^{i}(\mu))\subset F^{i}(\mu)$, then 
    \[f(tx_{0},e)=(tx_{0},\chi^{-1}_{E}\chi_{F}(t)\varphi(e))\] defines a continuous bundle morphism $f:X\times E\to X\times F$ by lemma ~\ref{cext}.\\
    $ii)$ If the bundles $\ME$ and $\mathcal{F}$ are isomorphic, then $\left.E \right.|_{\T}=\ME(\gamma_{I})\simeq \mathcal{F}(\gamma_{I})=\left.F\right.|_{\T}$. (Recall that $\gamma_{I}$ is the unique fixed point in $U_{I}$ for the top dimensional cone $I\in \Si^{(n)}$ and $\ME(\gamma_{I})$ denotes the fiber of $\ME$ over $\gamma_{I}$.) Conversely, if $\left.E \right.|_{\T}=\left.F\right.|_{\T}$, then the posets of subspaces $\{E^{i}(\mu)|i\in I, \mu\in \MR\}$ and $\{F^{i}(\mu)|i\in I, \mu\in \MR\}$ are isomorphic. By $i)$, the bundles $\ME$ and $\mathcal{F}$ are isomorphic.
\end{proof}

\begin{defi} \label{TKVS}
For a topological toric manifold $X$, the category $\TKVS_{X}$ is the category whose objects are vector spaces $E$ with a poset $P$ of subspaces $E^{i}(\mu)$ of $E$ indexed by $\mu\in \MR,i\in \Si^{(1)}$ which satisfy the following compatibility condition:\\
    For any $I\in \Si$, there exists a $\MR^{n}/\bp$-grading   \[E=\bigoplus_{[\chi]\in \MR^{n}/\bp}E_{[\chi]}\] for which
    \[E^{i}(\mu)=\sum_{\langle \chi,\beta_{i} \rangle {\geq_{c}}\mu}E_{[\chi]},\hspace{0.2cm} \forall i\in I.\]
    A morphism $f:(E,\{E^{i}(\mu)\})\to (F,\{F^{i}(\mu)\})$ in $\TKVS_{X}$ is a linear transformation $f:E\to F$ such that $f(E^{i}(\mu))\subset F^{i}(\mu)$ for all $i\in \Si^{(1)}$ and $\mu\in \MR$. $f$ is an isomorphism if $E\to F$ is an isomorphism of complex vector spaces. $f$ is a monomorphism if $E\to F$ is a monomorphism of complex vector spaces. $f$ is an epimorphism if $E\to F$ is an epimorphism of complex vector spaces and $f(E^{i}(\mu))=F^{i}(\mu)$ for all $i\in \Si^{(1)}$ and $\mu\in \MR$.
\end{defi}

\begin{theo} \label{main}
    The category $\TKVB_{X}$ of topological Klyachko vector bundles over the topological toric manifold $X=X(\De)$ is equivalent to the category $\TKVS_{X}$.
\end{theo}

\begin{coro}
    Let $\Phi:\TKVB_{X}\to \TKVS_{X}$ be an equivalence in the proof of Theorem ~\ref{main}, then a morphism $f:\ME\to \MF$ in $\TKVB_{X}$ is a monomorphism (resp.~epimorphism, resp.~isomorphism) if and only if $\Phi(f)$ is a monomorphism (resp.~epimorphism, resp.~isomorphism). 
\end{coro}

\begin{defi}
For each $I\in \Si$, we denote by $\Hom_{cont}(\T_{I},\CS)$ the group of continuous homomorphisms $\T_{I}\to \CS$ and by $\Hom_{sm}(\T_{I},\CS)$ the group of smooth homomorphisms $\T_{I}\to \CS$.   
\end{defi}

The following lemma explains the meaning of the index set $\MR^{n}/\bp$ in the grading of $E$.
\begin{lemm} \label{character}
For each $I\in \Si,\Hom_{cont}(\T_{I},\CS)=\Hom_{sm}(\T_{I},\CS)\simeq \MR^{n}/\bp$ as abelian groups.
\end{lemm}

\begin{proof}
Notice that $\T_{I}\simeq (\CS)^{k}$ where $k=|I|$, therefore by the proof of lemma ~\ref{allcharacters}, $\Hom_{cont}(\T_{I},\CS)=\Hom_{sm}(\T_{I},\CS)$. For the isomorphism, on the one hand, given $\alpha\in \MR^{n}$, we have $\chi^{\alpha}:\T\to \CS$ such that $\chi^{\alpha}(\la(h_1,\dots,h_m))=\prod_{k=1}^{m}h_{k}^{\langle \alpha,\beta_k \rangle}$. In particular, if $\alpha\in \bp$, then $\chi^{\alpha}(\T_{I})=\1$. Therefore by restricting to $\T_{I}$, $\alpha\in \MR^{n}/\bp$ gives rise to $\chi^{\alpha}:\T_{I}\to \CS$. On the other hand, given $\chi:\T_{I}\to \CS$ where $I\in \Si$, choose an $I_{0}\in \Si^{(n)}$ that contains $I$. Suppose $\chi(\la(h_1,\dots,h_m))=\prod_{k\in I}h_{k}^{\la_{k}}$ where $[h_1,\dots,h_m]\in \T_{I}$ (the bracket denotes the equivalence class). Let $\alpha=\Si_{k\in I}\la_{k}\alpha^{I_{0}}_{k}$ where $\alpha^{I_{0}}$ is dual to $\beta_{I_{0}}$, then $\chi=\chi^{\alpha}:\T_{I}\to \CS$. It is easy to check that these two assignments are inverse to each other, hence the result.
\end{proof}

\begin{lemm} \label{split}
There is a canonical splitting $\T\simeq \T_{I}\times \iota_{I}(\MO(I))$ where $\iota_{I}:\MO(I)\to \T$ is an embedding for each $I\in \Si^{(n)}$.
\end{lemm}

\begin{proof}
Given $I\in \Si$, pick an arbitrary $I_{0}\in \Si^{(n)}$ that contains $I$. We have the following exact sequence 
\[\1\to \T_{I}\to \T=(\CS)^{n}\simeq (\CS)^{m}/\ker\la \to \MO(I) \to \1. \]
The morphism $f_{1}:\T_I\to \T$ is $f_{1}([h_1,\dots,h_m])=[w]$ where $w_i=h_i$ if $i\in I$ and $w_i=1$ if $i\in I_{0}\backslash I$. The morphism $f_{2}:(\CS)^{I_{0}} \to \MO(I)$ is defined as follows: given $g=(g_{i})_{i\in I_{0}}\in (\CS)^{I_{0}}$, let $h_{i}=\chi^{\alpha^{I_{0}}_{i}}(g)$ for $i\in I_{0}$ and $h_{i}=1$ otherwise, then $\la(h)=g$. Define $f_{2}(g)=[w]$ where $w_i=0$ for $i\in I$ and $w_i=h_i$ for $i\in \Si^{(1)}\backslash I$. It is easy to check that the above sequence is exact. We define $\iota_{I}$ below: Given $w\in \MO(I)$, let $h\in 0^{I}\times 1^{\Si^{(1)}\backslash I}$ such that $[h]=w$, then let $\iota_{I}(w)_{i}=1$ if $i\in I$ and $\iota_{I}(w)_{i}=h_i$ otherwise. Then $f_{2}\circ \iota=id$. Hence the exact sequence splits.
\end{proof}

\begin{proof}[Proof of Theorem ~\ref{main}]
Klyachko's original proof goes through with the previous lemmas:\\
$ $\\
I. Construction of a functor $\Phi:\TKVB_{X}\to \TKVS_{X}$:\\
Let $\ME$ be a smooth Klyachko vector bundle over $X$. Let $X=\bigcup_{I\in \Si}U_{I}$ be the open covering by the open invariant subsets $U_{I}$ corresponding to $I\in \Si$. By definition, we have $\left.\ME\right|_{U_{I}}\simeq U_{I}\times \ME(\gamma_{I})$ for all $I\in \Si$. Because $U_{I}$ is $\T_{I}$-invariant and $\T$-invariant, we have 
\[\varphi_{I}:\T_{I}\to \Aut(\ME(\gamma_{I}))\]
which can be extended to 
\[\varphi_{I}:\T\to \Aut(\ME(\gamma_{I})) \]
by lemma ~\ref{split}.
The transition functions $f_{I|J}:U_{I}\bigcap U_{J}\to \Hom(\ME(\gamma_I),\ME(\gamma_J))$ satisfy the cocycle conditions $f_{I|J}f_{J|K}f_{K|I}=\1$ and $f_{I|J}f_{J|I}=\1$ and the equivariance condition 
\[f_{I|J}(tx)=\varphi_{I}(t)f_{I|J}(x)\varphi_{J}(t)^{-1}.\]
Let $x_{0}$ be an arbitrary point in the open dense orbit $\MO(\emptyset)$. We identify all fibers of $\ME$ with $E=\ME(x_{0})$ via $f_{IJ}$, i.e. we assume that $f_{IJ}(x_{0})=\1$. Then $f_{IJ}(tx_{0})=\varphi_{I}(t)\varphi_{J}(t)^{-1}$ and the map $\varphi_{I}(t)\varphi_{J}(t)^{-1}$ can be smoothly extended from $\T$ to $U_{I}\bigcap U_{J}$.
We will interpret this extendability condition in terms of inclusion of filtrations below. First, as a representation space of $\T_{I}$, the space $E$ admits a weight decomposition 
\[E=\bigoplus_{[\chi]\in \MR^{n}/\bp}E_{[\chi]}  \]
where $\chi\in \MR^{n}/\bp$ is considered as a map $\T_{I}\to \CS$ by lemma ~\ref{character} and $E_{[\chi]}:=\{x\in E| t\cdot x=\chi(t)x \hspace{0.2cm} \forall t\in \T_{I}\}$. Define a family of subspaces of $E$ by 
\begin{equation*}
E^{i,I}(\mu)=\sum_{[\chi]\in \MR^{n}/\bp,\langle \chi,\beta_{i} \rangle\geq_{c} \mu}E_{[\chi]},\hspace{0.2cm} \forall i\in I.     
\end{equation*}

We prove that $\varphi_{I}(t)\varphi_{J}(t)^{-1}$ can be continuously extended from $\T$ to $U_{I}\bigcap U_{J}$ iff $E^{i,J}(\mu)\subset E^{i,I}(\mu)$ for all $i\in I\bigcap J$ and $\mu\in \MR$. Indeed, let $e_i,i\in \La_{I}$ be a diagonalizing basis of the representation $\varphi_{I}$, i.e. $\varphi_{I}(t)e_{i}=\chi_{i}(t)e_{i}, \chi_{i}\in \Hom(\T,\CS)$ and let $f_{j},j\in \La_{J}$ be a diagonalizing basis of the representation $\varphi_{J}$, i.e. $\varphi_{J}(t)f_{j}=\psi_{j}(t)f_{j},\psi_{j}\in \Hom(\T,\CS)$. Let $f_{j}=\Si_{i}a_{ij}e_{i}$. Then 
\[\varphi_{I}\varphi_{J}^{-1}f_{j}=\Si_{i}a_{ij}\chi_{i}\psi_{j}^{-1}e_{i}. \]
By lemma ~\ref{cext}, the character $\chi_{i}\psi_{j}^{-1}$ can be continuously extended to $U_{I}\bigcap U_{J}$ iff $\langle \chi_{i},\beta_{k}\rangle \geq_{c} \langle \psi_{j},\beta_{k}\rangle $ for all $k\in I\bigcap J$. In other words, 
\[E^{i,J}(\mu)\subset E^{i,I}(\mu) \hspace{0.2cm} \forall i\in I\bigcap J \text{ and } \mu\in \MR.\] As a consequence, $E^{i,I}(\mu)=E^{i,J}(\mu)$ for all $i\in I,J$ because both $f_{IJ}$ and $f_{JI}$ can be continuously extended to $U_{I}\bigcap U_{J}$. Therefore we denote $E^{i,I}(\mu)$ simply by $E^{i}(\mu)$. Define $\Phi(\ME):= (E,\{E^{i}(\mu)| i\in \Si^{(1)},\mu\in \MR\})$. \\
$ $\\
Suppose $f:\ME\to \MF$ is a $\T$-equivariant map between two topological Klyachko vector bundles $\ME,\MF$ over $X$. For each $I\in \Si^{(n)}$, we have the following commutative diagram 
\begin{center}
    \begin{tikzcd}
        \left.\ME\right.\!|_{U_{I}} \ar[r,"f"] \ar[d,"\simeq"] & \left.\MF \right.\!|_{U_{I}} \ar[d,"\simeq"]\\
        U_{I}\times E \ar[r] &U_{I}\times F
    \end{tikzcd}
\end{center}
Define $\Phi(f)$ to be $\left.f\right.|_{x_{0}}:E\to F$. Suppose $\{E^{i}(\mu)| i\in \Si^{(1)},\mu\in \MR\}$ comes from a grading $E=\bigoplus_{\chi\in \MR^{n}/\bp}E_{[\chi]}$ and $\{F^{i}(\mu)| i\in \Si^{(1)},\mu\in \MR\}$  comes from $F=\bigoplus_{\chi\in \MR^{n}/\bp}F_{[\chi]}$, then by equivariance of $f$, we must have $\Phi(f)(E^{i}(\mu))\subset F^{i}(\mu)$.\\
$ $\\
II. Construction of a functor $\Psi: \TKVS_{X}\to \TKVB_{X}$: \\
Given a family of doubly indexed subspaces $E^{i}(\mu)$ that correspond to $E=\bigoplus_{\chi\in \MR^{n}/\bp}E_{[\chi]}$, we have a representation of $\T_{I}$ that can be extended to a representation $\varphi_{I}$ of $\T$ by lemma ~\ref{split}. As the subspaces are decreasing, $f_{IJ}(tx_{0})=\varphi_{I}(t)\varphi_{J}^{-1}(t)$ can be extended to $U_{I}\bigcap U_{J}$ by the first part of this proof. Therefore the cocycle $f_{IJ}$ determines a topological Klyachko vector bundle $\ME=\Psi(E,\{E^{i}(\mu)| i\in \Si^{(1)},\mu\in \MR\})$ over $X(\De)$. By lemma ~\ref{indext}, $\left.\ME \right.|_{U_{I}}$ does not depend on the choice of an extension of $\varphi_{I}$ for each $I\in \Si^{(n)}$. That the bundle $\ME$ does not depend on the choice of gradings $E^{I}$ is exactly the same as Klyachko's original proof.\\
$ $\\
Suppose given a morphism $f:(E,\{E^{i}(\mu)\})\to (F,\{F^{i}(\mu)\})$ in $\TKVS_{X}$, then by lemma ~\ref{indext}, we have a $\T$-equivariant morphism of vector bundles $f_{I}:U_{I}\times E\to U_{I}\times F$ for each $I\in \Si^{(n)}$. For any two $I,J\in \Si^{(n)}$, notice that $f_{I}$ and $f_{J}$ coincide as their restrictions to the dense subset $\T\bigcap U_{I}\bigcap U_{J}$ are the same. Hence they glue to a $\T$-equivariant morphism $\Psi(E,\{E^{i}(\mu)\})\to \Psi(F,\{F^{i}(\mu)\})$. \\
$ $\\
We can quickly check that $\Psi(\Phi(\ME))\simeq \ME$ and $\Phi(\Psi(E,\{E^{i}(\mu)\}))\simeq (E,\{E^{i}(\mu)\})$.
\end{proof}

\section{Smooth Klyachko vector bundles}
\begin{defi}
A smooth Klyachko vector bundle of rank $r$ over an $n$-dimensional topological toric manifold $X(\De)$ with a $\T$-equivariant complex structure associated with a topological toric fan $\De=(\Si,\beta)$ is a $\T$-equivariant smooth complex vector bundle $\ME\to X(\De)$ whose restriction $\left. \ME \right|_{U_{I}}$ is equivariantly diffeomorphic to $U_{I}\times \C^{r}$ for all $I\in \Si$ and whose action map $\T\times \ME\to \ME$ is smooth.
\end{defi}

\begin{lemm} \label{ext2}
A smooth character $\chi^{\alpha}:\CS\to \CS$ where $\alpha=(b+c\sqrt{-1},v)\in \MR$ can be smoothly extended to $\C$ iff $c=0, b\in \N$ and $b\pm v\in 2\N$.  
\end{lemm}

\begin{proof}
     Write $i$ for $\sqrt{-1}$ and $g=\rho e^{i\theta}\in \CS=x+yi$ where $\rho>0,\theta\in [0,2\pi),x,y\in \R$. The sufficiency is clear. We prove necessity below. Given $\chi^{\alpha}:\CS\to \CS$ that can be smoothly extended to $\C$ where $\alpha=(b+ci,v)$, by lemma ~\ref{ext1}, $b>0$ or $b=c=v=0$. In the latter case, obviously we have $c=0, b\in \N$ and $b\pm v\in 2\N$. Suppose $b>0$. Let $f(\rho):=\left.\chi^{\alpha}\right |_{\R}(\rho)=\rho^{b+ci}$. We first prove that $b\in \N$. Suppose not, then there exists $n_{0}\in \N$ such that $n_{0}<b<n_{0}+1$ and $\frac{d^{n_{0}+1}f}{d\rho^{n_{0}+1}}=C_{1}\rho^{b-n_{0}-1+ci}$ ($C_{1}$ is some non-zero constant) diverges as $\rho\to 0^{+}$, which contradicts that $\chi^{\alpha}$ is smooth at $0$. Therefore $b\in \N$. Next we prove that $c=0$. Suppose $c\neq 0$, then $\frac{d^{b+1}f}{d\rho^{b+1}}=C_{2}\cdot ci\cdot \rho^{-1+ci}$ ($C_{2}$ is some non-zero constant) diverges as $\rho\to 0^{+}$. Therefore $c=0$. Now we prove that $b\pm v \in 2\N$. The proof below is based on the observation that for Puiseux series $F,G$ in one indeterminate $x$ with finitely many terms, $deg((\frac{F}{G})')\leq deg(\frac{F}{G})$. Here the degree of a Puiseux series with finitely many terms is the largest exponent whose corresponding coefficient is nonzero. We consider three cases: 1) $0\leq v\leq b$, 2) $0<b<v$, 3) $v<0<b$. In case 1, $\chi^{\alpha}(g)=g^{v}\rho^{b-v}=(\sum_{k=0}^{v}\binom{n}{k}x^{k}(yi)^{v-k})(x^{2}+y^{2})^{\frac{b-v}{2}}$. By separating the real and imaginary parts, we see that $x^{\alpha}$ is smooth if and only if $b-v\in 2\N$ if and only if $b\pm v\in 2\N$. In case 2, $\chi^{\alpha}(g)=\frac{g^{v}}{\rho^{v-b}}$. Its real part is $\frac{F}{G}$ where $F(x,y)=\sum_{k=0}^{[\frac{v}{2}]}(-1)^{k}x^{v-2k}y^{2k}$ and $G(x,y)=(x^{2}+y^{2})^{\frac{v-b}{2}}$. Since $degF=v, degG=v-b$, by the observation, $\lim_{\rho\to 0^{+}}\frac{\partial ^{b+1}\frac{F}{G}}{\partial x^{b+1}}$ does not exist. In case 3, $\chi^{\alpha}(g)=\rho^{b+v}\OL{g}^{-v}$. By a similar reason as in case 1, $\chi^{\alpha}(g)$ is smooth if and only $b+v\in 2\N$ if and only if $b\pm v\in 2\N$. Hence we proved the necessity.
\end{proof}

\begin{defi} \label{geqs}
We can define a partial order $\geq_{s}$ on $\MR$ as follows: For $\alpha=(b+c\sqrt{-1},v)\in \MR$, we say $\alpha \geq_{s} 0=(0+0\sqrt{-1},0)$ if $c=0, b\in \N, b\pm v\in 2\N$. For $\alpha,\beta\in \MR$, we say $\alpha\geq_{s}\beta$ if $\alpha-\beta\geq_{s} 0$.
\end{defi}

\begin{prop}
    $\geq_{s}$ is a partial order.
\end{prop}

\begin{proof}
Reflexivity and anti-symmetry are clear. Transitivity: Suppose $\alpha_{i}=(b_{i}+c_{i}\sqrt{-1},v_{i})\in \MR,i=1,2,3$ and $\alpha_{1}\geq_{c}\alpha_{2},\alpha_{2}\geq_{c}\alpha_{3}$, i.e. $b_{1}-b_{2},b_{2}-b_{3}\in \N, (b_{1}-b_{2})\pm (v_{1}-v_{2}),(b_{2}-b_{3})\pm (v_{2}-v_{3})\in 2\N$, then by adding these expressions, we have $b_{1}-b_{3}\in \N, (b_{1}-b_{3})\pm(v_{1}-v_{3})\in 2\N$.
\end{proof}

\begin{lemm} \label{smext}
 If $I\in \Si$, then $\chi^{\alpha}:\T\to \CS, \alpha\in \MR^{n}$ can be smoothly extended to $U_{I}$ iff $\langle \alpha,\bi\rangle \geq_{s} 0 \hspace{0.2cm}\forall i\in I$.  
\end{lemm}

\begin{proof}
    The proof is the same as lemma ~\ref{cext}.
\end{proof}

\begin{lemm} \label{indext2}
\leavevmode
Let $X=U_{I}$ be an affine chart of a topological tori manifold corresponding to a cone $I\in \Si^{(n)}$. 
\begin{enumerate}[label=\roman*),]
    \item Define a family of subspaces of a complex $\T$-module $E$ by 
    \[E^{i}(\mu)=\bigoplus_{\langle \chi,i\rangle \geq_{s}\mu} E_{\chi},i\in I, \] where $E_{\chi}\subset E$ is the isotypical component of the smooth character $\chi\in \MR^{n}$.If $F$ is another $\T$-module, we define $F^{i}(\mu)$ similarly. Then the space of equivariant homomorphisms $\Hom_{\T}(X\times E,X\times F)$ is canonically isomorphic to the space of order preserving linear operators $\varphi:E\to F$ 
    \item Two smooth Klyachko vector bundles $X\times E$ and $X\times F$ are isomorphic if and only if the restrictions of the smooth representations are isomorphic $\left.E\right.|_{\T_{I}}\simeq \left.F\right.|_{\T_{I}}$.
\end{enumerate}   
\end{lemm}

\begin{proof}
The proof is same as the one for lemma ~\ref{indext} except that we replace lemma ~\ref{cext} by lemma ~\ref{smext}.    
\end{proof}

\begin{defi}
    For a topological toric manifold $X$, the category $\SKVB_{X}$ is the category whose objects are smooth Klyachko vector bundles over $X$ and whose morphisms are morphisms of equivariant complex vector bundles that are $\C$-linear in fibers.
\end{defi}

\begin{defi} \label{SKVS}
For a topological toric manifold $X$, the category $\SKVS_{X}$ is the category whose objects are vector spaces $E$ with a poset $P$ of subspaces $E^{i}(\mu)$ of $E$ indexed by $\mu\in \MR,i\in \Si^{(1)}$ which satisfy the following compatibility condition:\\
    For any $I\in \Si$, there exists a $\MR^{n}/\bp$-grading   \[E=\bigoplus_{[\chi]\in \MR^{n}/\bp}E_{[\chi]}\] for which
    \[E^{i}(\mu)=\sum_{\langle \chi,\beta_{i} \rangle \geq_{s}\mu}E_{[\chi]},\hspace{0.2cm} \forall i\in I.\]
    A morphism $f:(E,\{E^{i}(\mu)\})\to (F,\{F^{i}(\mu)\})$ in $\SKVS_{X}$ is a linear transformation $f:E\to F$ such that $f(E^{i}(\mu))\subset F^{i}(\mu)$ for all $i\in \Si^{(1)}$ and $\mu\in \MR$. $f$ is an isomorphism if $E\to F$ is an isomorphism of complex vector spaces. $f$ is a monomorphism if $E\to F$ is a monomorphism of complex vector spaces. $f$ is an epimorphism if $E\to F$ is an epimorphism of complex vector spaces and $f(E^{i}(\mu))=F^{i}(\mu)$ for all $i\in \Si^{(1)}$ and $\mu\in \MR$.
\end{defi}

Following the proofs in Section 4, we can prove:

\begin{theo} \label{main2}
    The category $\SKVB_{X}$ of smooth Klyachko vector bundles over the topological toric manifold $X=X(\De)$ is equivalent to the category $\SKVS_{X}$. 
\end{theo}

\begin{coro}
    Let $\Phi:\SKVB_{X}\to \SKVS_{X}$ be an equivalence in the proof of Theorem ~\ref{main}, then a morphism $f:\ME\to \MF$ in $\SKVB_{X}$ is a monomorphism (resp.~epimorphism, resp.~isomorphism) if and only if $\Phi(f)$ is a monomorphism (resp.~epimorphism, resp.~isomorphism). 
\end{coro}

In the next section, we will consider a canonical exact sequence constructed in \cite{IFM12} of equivariant complex vector bundles where the morphisms are $\R$-linear in the fibers. This motivates to define the following categories.

\begin{defi} \label{SKVB-real}
    For a topological toric manifold $X$, the category $\widetilde{\SKVB_{X}}$ is the category whose objects are smooth Klyachko vector bundles over $X$ and whose morphisms are morphisms of equivariant complex vector bundles that are $\R$-linear in fibers.
\end{defi}

\begin{defi} \label{SKVS-real}
For a topological toric manifold $X$, the category $\widetilde{\SKVS_{X}}$ is the category having the same objects as $\SKVS_{X}$. 
    A morphism $f:(E,\{E^{i}(\mu)\})\to (F,\{F^{i}(\mu)\})$ in $\widetilde{\SKVS_{X}}$ is an $\R$-linear map $f:E\to F$ such that $f(E^{i}(\mu))\subset F^{i}(\mu)$ for all $i\in \Si^{(1)}$ and $\mu\in \MR$. $f$ is an isomorphism if $E\to F$ is an isomorphism of real vector spaces. $f$ is a monomorphism if $E\to F$ is a monomorphism of real vector spaces. $f$ is an epimorphism if $E\to F$ is an epimorphism of real vector spaces and $f(E^{i}(\mu))=F^{i}(\mu)$ for all $i\in \Si^{(1)}$ and $\mu\in \MR$.
\end{defi}

Below is a modified version of Theorem ~\ref{main2}. 
\begin{theo} \label{main3}
    The category $\widetilde{\SKVB_{X}}$ is equivalent to the category $\widetilde{\SKVS_{X}}$.
\end{theo}

{\bf Remark:} The morphisms in the category $\widetilde{\SKVB_{X}}$ are essentially morphisms between equivariant real vector bundles. We stick to complex vector bundles to guarantee that the fiber $E$ over $x_{0}$ can be decomposed into eigenspaces. In general, this decomposition may not exist if we consider real Klyachko vector bundles. 

\section{Applications and examples}
\subsection{A canonical exact sequence}
In ~\cite{IFM12}, the authors constructed a family of equivariant complex line bundles $L_{i}, 1\leq i \leq m,$ over $X(\De)$ and prove the following theorem.

\begin{theo} [Theorem 6.5 of \cite{IFM12}] \label{canonical}
    There is a canonical short exact sequence of $(\CS)^{n}$-equivariant complex vector bundles 
    \[0\to X(\De)\times \Lie(\Ker\lambda)\to \bigoplus_{i=1}^{m} L_{i}\stackrel{f}\longrightarrow  \tau X(\De)\to 0 \]
    where  $(\CS)^{n}$ acts on the fibers of $X(\De)\times \Lie(\Ker\lambda)$ trivially and $\tau X(\De)$ denotes the tangent bundle. 
\end{theo}
{\bf Remark:} $f$ is a map between equivariant complex vector bundles that is only $\R$-linear in the fibers, which makes the computation of the fiber maps below very complicated. \\
$ $\\
In this section, we will recall the construction of $L_{i}$, prove that $X(\De)\times \Lie(\Ker\lambda)$, $ \bigoplus_{i=1}^{m} L_{i}, \tau X(\De)$ are all smooth Klyachko vector bundles and give another proof of the exactness of the sequence in Theorem ~\ref{canonical} using the poset data of these bundles.\\
$ $\\
{\bf Construction of $L_{i}$:} Let $\pi_{i}:(\CS)^{m}\to \CS$ be the projection onto the $i$-th factor of $(\CS)^{m}$ and denote by $V(\pi_{i})$ the complex one-dimensional representation space of $\pi_{i}$. Let $(\CS)^{m}$ act on $U(\Si)\times V(\pi_{i})$ by 
\[(g_{1},\cdots,g_{m})((z_1,\cdots,z_{m}),w)=((g_1z_1,\cdots,g_mz_m),g_{i}w) \]
for $(g_1,\cdots,g_m)\in (\CS)^{m}, (z_1,\cdots,z_m)\in U(\Si), w\in V(\pi_{i})$. The $(\CS)^{m}$-equivariant complex line bundle $L_{i}$ is 
\[L_{i}: (U(\Si)\times V(\pi_{i}))/\Ker \lambda\to U(\Si)/\Ker\lambda=X(\De).\]
\noindent 
{\bf Equivariant trivializations:}
$X(\De)=\bigcup_{I\in \Si^{(n)}}U_{I}$. For each $I\in \Si^{(n)}$, on $U_{I}$, depending on whether $i\in I$, we have two types of equivariant trivializations.\\
When $i\in I$, the equivariant trivialization \[\psi_{I}:(U(I)\times V(\pi_{i}))/\Ker\lambda\to U_{I}\times \C\] is given by 
\begin{equation}\label{tri1}
  [(z,w)]\mapsto ([z],w\cdot \prod_{k\notin I}z_{k}^{\langle \alpha^{I}_{i},\beta_{k}\rangle}).
\end{equation}
When $i\notin I$, the equivariant trivialization \[\psi_{I}:(U(I)\times V(\pi_{i}))/\Ker\lambda\to U_{I}\times \C\] is given by 
\begin{equation}\label{tri2}
     [(z,w)]\mapsto ([z],\frac{w}{z_{i}}).
\end{equation}\\
{\bf Injectivity of $\psi_{I}$:}
Suppose $\psi_{I}([z,w])=\psi_{I}([z',w'])$. \\
Case 1: $i\in I. \,[z]=[z'], w\cdot \prod_{k\notin I}z_{k}^{\langle \alpha^{I}_{i},\beta_{k}\rangle}=w'\cdot \prod_{k\notin I}z_{k}^{\prime \langle \alpha^{I}_{i},\beta_{k}\rangle}$, therefore $z_{k}'=hz_{k}$ for some $h\in \Ker\lambda$ and $w'=w\prod_{k\notin I}h^{\langle \alpha^{I}_{i},\beta_{k}\rangle}=w\cdot h$ where the last equality follow from ~\ref{kerl}. Hence $[z,w]=[z',w']$.\\
$ $\\
Case 2: $i\notin I. \, [z]=[z'],\frac{w}{z_{i}}=\frac{w'}{z'_{i}}$, therefore $z_{k}'=z_{k}$ for some $h\in \Ker\lambda$ and $w'=\frac{z_{i}'}{z_{i}}w=w\cdot h$. Hence $[z,w]=[z',w']$.\\

\noindent
{\bf Surjectivity of $\psi_{I}$:}
Suppose $(z,c)\in U_{I}\times \C$.\\
Case 1: $i\in I$. Let $w=c\cdot \prod_{k\notin I}z_{k}^{-\langle \alpha^{I}_{i},\beta_{k}\rangle}$, then $\psi_{I}([z,w])=(z,c)$.\\
$ $\\
Case 2: $i\notin I$. Let $w=c\cdot z_{i}$, then $\psi_{I}([z,w])=(z,c)$.\\

\noindent
{\bf Equivariance of $\psi_{I}:$}\\
Case 1: $i\in I$. Let $\T$ act on $U_{I}\times \C$ via 
\[g\cdot (z,c)=(g\cdot z,\chi^{\alpha^{I}_{i}}(g)c), g\in \T, (z,c)\in U_{I}\times \C. \]
Given $g\in \T, ([z],c)\in U_{I}\times \C,$ let $h=(h_1,\cdots,h_m)\in (\CS)^{m}$ where 
\[h_k=\begin{cases} \chi^{\alpha^I_k}(g) \quad&\text{for $k\in I$},\\
1 \quad&\text{for $k\notin I$.}\end{cases}\]
Then $\lambda(h)=g$. 

\begin{align*}
\psi_{I}(g\cdot[z,w])
&= \psi_{I}([(h_{k}z_{k})_{k\in I},h_{i}w])\\
&= (g\cdot [z],h_{i}w\prod_{k\notin I}(h_{k}z_{k})^{\langle \alpha^{I}_{i},\beta_{k}\rangle})\\
&= (g\cdot [z],h_{i}w\prod_{k\notin I}z_{k}^{\langle \alpha^{I}_{i},\beta_{k}\rangle})\\
&= g\cdot \psi_{I}([z,w]).
\end{align*}
Case 2: $i\notin I$. Let $\T$ act on $U_{I}\times \C$ via 
\[g\cdot ([z],c)=([g\cdot z],c), g\in \T, (z,c)\in U_{I}\times \C. \]
Given $g\in \T, (z,c)\in U_{I}\times \C,$ choose $h$ as above,
\begin{align*}
\psi_{I}(g\cdot[z,w])
&= \psi_{I}([(h_{k}z_{k})_{k\in I},h_{i}w])\\
&= (g\cdot [z],\frac{h_{i}w}{h_{i}z_{i}})\\
&= (g\cdot [z],\frac{w}{z_{i}})\\
&= g\cdot \psi_{I}([z,w]).
\end{align*}
\\
\noindent
{\bf Transition functions $g_{IJ}$ of $L_{i}$:}
Given $[z]\in U_{I}\bigcap U_{J} ,I,J\in \Si^{(n)}$, there are three cases: \\
Case 1: $i\in I\bigcap J$. \[\psi_{I} \circ\psi_{J}^{-1}([z],c)=([z],c\prod_{k\notin I}z_{k}^{\langle \alpha^{I}_{i}-\alpha^{J}_{i},\beta_{k} \rangle}),\]
therefore \[ g_{IJ}([z])=\prod_{k\notin I}z_{k}^{\langle \alpha^{I}_{i}-\alpha^{J}_{i},\beta_{k} \rangle}. \]\\
Case 2: $i\in I, i\notin J$. \[\psi_{I} \circ\psi_{J}^{-1}([z],c)=([z], \frac{c}{z_{i}\prod_{k\notin I}z_{k}^{\langle \alpha^{I}_{i},\beta_{k}\rangle}})=([z],\frac{c}{\varphi_{I}([z])}). \] 
Therefore \[ g_{IJ}([z])=\varphi^{-1}_{I}([z]). \] \\
Case 3: $i\notin I, i\in J$. This is essentially the same as case 2. \\
Hence $L_{i}$ is a smooth Klyachko vector bundle.\\
$ $\\
\noindent
{\bf Gradings of $L_{i}$:}
Let $x_{0}$ be the identity of $\T\subset X(\De)$ and $E_{i}\simeq \C$ the fiber of $L_{i}$ over $x_{0}$. For $I\in \Si^{(n)}$, 
\[ E_{i}=\bigoplus_{[\chi]\in \MR^{n}/\bp}{E_{i}}_{[\chi]}=\begin{cases}
    {E_{i}}_{[\chi^{\alpha^{I}_{i}}]} \quad i\in I,\\
    {E_{i}}_{[\chi^{0}]} \quad i\notin I.
\end{cases} \]\\
$ $\\
\noindent
{\bf Poset $P_{i}$ associated with $L_{i}$:}\\
For $I\in \Si^{(n)}$ containing $i$ and $k\in I$:
\[E_{i}^{k}(\mu)=\sum_{\langle \chi,\beta_{k}\rangle \geq_{s}\mu}{E_{i}}_{[\chi]}=\begin{cases}
    E_{i}, \quad \langle \alpha^{I}_{i},\beta_{k}\rangle=\delta_{ik}\1 \geq_{s} \mu,\\
    0, \quad \text{otherwise}.
\end{cases} \]
For $I\in \Si^{(n)}$ not containing $i$ and $k\in I$:
\[E_{i}^{k}(\mu)=\sum_{\langle \chi,\beta_{k}\rangle \geq_{s}\mu}{E_{i}}_{[\chi]}=\begin{cases}
    E_{i}, \quad 0 \geq_{s} \mu,\\
    0, \quad \text{otherwise}.
\end{cases} \]
\noindent
{\bf Poset $P$ associated with $\bigoplus_{i=1}^{m}L_{i}$:}\\
Let $E=\bigoplus_{i=1}^{m}E_{i}\simeq \C^{m}$ be the fiber of $\bigoplus_{i=1}^{m}L_{i}$ over $x_{0}$. For $I\in \Si^{(n)}, k\in I$, 
\[ E^{k}(\mu)=\bigoplus_{i=1}^{m}E^{k}_{i}(\mu)=\begin{cases}
   E, \quad &0\geq_{s}\mu \\
   E_{k}, \quad &1\geq_{s}\mu, 0\ngeq_{s} \mu\\
   0, &1\ngeq_{s} \mu, 0\ngeq_{s}\mu
\end{cases} \]

$ $\\
\noindent
{\bf Poset $P_{c}$ associated with $X(\De)\times F$ with $\T$ acting trivially on $F$:}\\
$F=F_{[\chi^{0}]}$ is a grading of the $\T$-module $F$. \\
For $I\in \Si^{(n)},k\in I$, 
\[F^{k}(\mu)=\begin{cases}
    F, \quad 0\geq_{s}\mu,\\
    0, \quad \text{otherwise}.
\end{cases} \]
$ $\\
\noindent
{\bf Poset $P_{\tau}$ associated with $\tau X(\De)$:}
Let $T\simeq \C^{n}$ be the fiber of $\tau X(\De)$ over $x_{0}$.\\
For $I\in \Si^{(n)}$, via the local chart 
\[\varphi_{I}:U_{I}\to \bigoplus_{i\in I}V(\chi^{\alpha^{I}_{i}})=\C^{I}, \]
we get the isomorphism of equivariant vector bundles
\begin{center}
\begin{tikzcd}
    \left.\tau X(\De)\right.\!\!|_{U_{I}}\ar[r] \ar[d] & \C^{I}\times \C^{I} \ar[d,"pr_{1}"]\\
    U_{I} \ar[r,"\varphi_{I}"] &\C^{I}
\end{tikzcd}    
\end{center}
Therefore we get a grading of $T$:
\[ T=\bigoplus_{i\in I}T_{[\chi^{\alpha^{I}_{i}}]} \]
For $I\in \Si^{(n)}$ and $k\in I$,
\[T^{k}(\mu)=\sum_{\langle \chi^{\alpha^{I}_{i}},\beta_{k}\rangle =\delta_{ik}1\geq_{s}\mu}T_{[\chi^{\alpha^{I}_{i}}]}  \]
\begin{proof} [Proof of Theorem ~\ref{canonical}]
We will do the following:\\
1. Show that the poset map $f_{x_{0}}: P\to P_{\tau}$ corresponding to $\bigoplus_{i=1}^{m} L_{i}\to \tau X(\De)$ in \cite{IFM12} is surjective, which shows that $\bigoplus_{i=1}^{m} L_{i}\to \tau X(\De)$ is surjective. \\
2. Prove that the poset $\{F^{k}(\mu):=\ker(f^{k}_{\mu})| k\in \Si^{(1)}, \mu\in \MR\}$ is the poset associated with $X(\De)\times F$ with $F=\Lie(\Ker\lambda)$, which shows that $X(\De)\times F$ is the kernel of $\bigoplus_{i=1}^{m} L_{i}\to \tau X(\De)$. Here $f^{k}_{\mu}:E^{k}(\mu)\to T^{k}(\mu)$ is the restriction of $f_{x_{0}}:E\to T$. \\
$ $\\
\noindent
{\bf 1. Surjectivity of $f_{x_{0}}: P\to P_{\tau}$:}\\
The quotient map \[q:U(\Si)\to U(\Si)/\Ker\lambda=X(\De)\]
induces 
\[ \tau U(\Si)=U(\Si)\times \C^{m} \stackrel{dq}\longrightarrow \tau X(\De). \]
Since $\Ker\lambda$ acts on $\tau X(\De)$ trivially, by passing to the quotient, we get \[(U(\Si)\times \C^{m})/\Ker\lambda=\bigoplus_{i=1}^{m} L_{i}\stackrel{f}\longrightarrow \tau X(\De). \]
To find $f_{x_{0}}$, we need to compute $f$ in local charts. Given $I\in \Si^{(n)}$, we have the following commutative diagram
\begin{center}
    \begin{tikzcd}
       U(I)\times \R^{2m} \ar[rdd,"dq_{I}"] \ar[d,"\simeq"] &\\
        U(I)\times \C^{m} \ar[d,"\pi"]  &\\
        (U(I)\times \C^{m})/\Ker\lambda \ar[r,"f"] \ar[d,"\Phi_{I}"',"\simeq"] &\C^{I}\times (\R^{2})^{I}\\
        U_{I}\times \C^{m} \ar[d,"\simeq"] &\\
        U_{I}\times \R^{2m} \ar[uur,"f_{I}"']
    \end{tikzcd}
\end{center}
where $q_{I}=\varphi_{I}\circ q$ and $\Phi_{I}$ is induced by the local charts of $L_{i}$ given above. To specify the maps in this diagram, we use the following ordered bases for the tangent spaces: Let $z_{j}=x_{j}+iy_{j},j=1,\cdots,m$ be the complex coordinates of $U(I)$ and $(\frac{\partial}{\partial x_{j}})_{j\in I},(\frac{\partial}{\partial x_{j}})_{j\notin I},(\frac{\partial}{\partial \overline{x}_{j}})_{j\in I},(\frac{\partial}{\partial \overline{x}_{j}})_{j\notin I}$ be the ordered local frame of $T_{\R}U(I)$ where inside $(\frac{\partial}{\partial x_{j}})_{j\in I}$ the vector fields $\frac{\partial}{\partial x_{j}}$ are ordered according to the natural order on $I$ and similarly on the other three packages. Similarly let $z'_{j}=x_{j}'+iy_{j}',j\in I$ be the complex coordinates of $\C^{I}$ and $(\frac{\partial}{\partial x_{j}'})_{j\in I},(\frac{\partial}{\partial y_{j}'})_{j\in I}$. Then 
\begin{center}
    \begin{tikzcd}
       (z,w_{\R}) \ar[rdd, maps to, "dq_{I}"] \ar[d,maps to] &\\
        (z,w_{\C}) \ar[d,maps to,"\pi"]  &\\
        {[z,w_{\C}]} \ar[r,maps to,"f"] \ar[d, maps to,"\Phi_{I}"'] &(q_{I}(z),J_{I}(z)w_{\R})\\
        {[z,v_{\C}]} \ar[d,maps to] &\\
        {[z,v_{\R}]}\ar[uur,maps to,"f_{I}"']
    \end{tikzcd}
\end{center}
where 
\begin{align*}
    w_{\R}&=(\Tilde{X}_{1},\cdots,\Tilde{X}_{m},\Tilde{Y}_{1},\cdots,\Tilde{Y}_{m})^{T}, &w_{\C}&=(w_{1},\cdots,w_{m})^{T}=(\Tilde{X}_{1}+i\Tilde{Y}_{1},\cdots, \Tilde{X}_{m}+i\Tilde{Y}_{m})^{T},\\
    v_{\R}&=(X_{1},\cdots,X_{m},Y_{1},\cdots,Y_{m})^{T}, &v_{\C}&=(v_{1},\cdots,v_{m})^{T}=(X_{1}+iY_{1},\cdots,X_{m}+iY_{m})^{T},\\
    v_{i}&=\begin{cases}
    w_{i}\prod_{l\notin I}z_{l}^{\langle \alpha^{I}_{i},\beta_{l}\rangle}, \quad &i\in I,\\
    \frac{w_{i}}{z_{i}}, \quad &i\notin I.
\end{cases} &
\end{align*}
Let $C(z)$ be the $m\times m$ diagonal matrix whose $(i,i)$ entry is $a_{ii}=\begin{cases}
    \frac{1}{\prod_{l\notin I}z_{l}^{\langle \alpha^{I}_{i},\beta_{l}\rangle}}, \quad &i\in I,\\
    z_{i}, \quad &i\notin I.
\end{cases}$, then $w_{\C}=C(z)v_{\C}$.
Notice that $C(x_{0})=I_{m}$, then following the diagram above, we have 
\[f_{x_{0}}(v_{\R})=J_{I}(x_{0})v_{\R}.\]
The differential $dq_{I}$ is $\R$-linear but the expression of $q_{I}$ is given in complex coordinates, so we pass to its complexification \[dq_{I}\otimes_{\R}\C:U(I)\times \C^{2m}\to \C^{I}\times (\C^{2})^{I}\] to compute the matrix form of $dq_{I}$.
Let $u_{j}(z)$ be the $j$-coordinate of $\varphi_{I}(z)$, i.e. $u_{j}(z)=z_{j}\prod_{l\notin I}z_{l}^{\langle \alpha^{I}_{j},\beta_{l}\rangle}$, then the matrix form $J_{I,\C}(z)$ of $dq_{I}\times_{\R}\C$ restricted to $z$ is 
\[J_{I,\C}(z)=
\begin{pmatrix}
    A(z) &B(z) \\
    \overline{B(z)} &\overline{A(z)}
\end{pmatrix}
\]
where $A(z)$ is the $n\times m$ matrix whose $(j,k)$-entry is $\frac{\partial u_{j}}{\partial z_{k}}(z), j\in I, 1\leq k\leq m$ and 
$B(z)$ is the $n\times m$ matrix whose $(j,k)$-entry is $\frac{\partial u_{j}}{\partial \overline{z}_{k}}(z), j\in I, 1\leq k\leq m$.
The matrix $J_{I}(z)$ is obtained from $J_{I,\C}(z)$ as
\begin{align*}
    J_{I}(z)&=\frac{1}{2}\begin{pmatrix}
    I_{n} &I_{n}\\
    -iI_{n} &iI_{n}
\end{pmatrix}J_{I,\C}(z)\begin{pmatrix}
    I_{m} &iI_{m} \\
    I_{m} &-iI_{m}
\end{pmatrix}\\
&=\frac{1}{2}\begin{pmatrix}
    A(z)+\overline{A(z)}+B(z)+\overline{B(z)} &i(A(z)-\overline{A(z)}-B(z)+\overline{B(z)})\\
    -i(A(z)+B(z))+i(\overline{A(z)}+\overline{B(z)}) &A+\overline{A(z)}-B(z)-\overline{B(z)}
\end{pmatrix}
\end{align*}
For each \(l\notin I\) write
\[
a_{jl} := \langle \alpha^I_j,\beta_l\rangle = (b_{jl}+i\,c_{jl},\,v_{jl}) \in \mathbb{C}\times\mathbb{Z}.
\]
We define the power by
\[
z^{a_{jl}}:= |z|^{\,b_{jl}}\,\exp\Bigl(i\,c_{jl}\,\ln|z|\Bigr)\left(\frac{z}{|z|}\right)^{v_{jl}},
\]
so that the function \(z\mapsto z^{a_{jl}}\) is smooth and complex–valued (after choosing smooth branches for \(\ln|z|\) and \(\arg z\)).

It is convenient to write
\[
z^{a_{jl}} = \exp\Bigl(F_{jl}(z,\overline{z})\Bigr),
\]
with
\[
F_{jl}(z,\overline{z}) = \frac{b_{jl}+i\,c_{jl}}{2}\,\ln(z\overline{z})
+\frac{v_{jl}}{2}\,\ln\Bigl(\frac{z}{\overline{z}}\Bigr).
\]
Then the chain rule gives
\[
\frac{\partial F_{jl}}{\partial z} = \frac{b_{jl}+i\,c_{jl}}{2}\cdot\frac{1}{z}+\frac{v_{jl}}{2}\cdot\frac{1}{z} 
=\frac{(b_{jl}+i\,c_{jl})+v_{jl}}{2z},
\]
and similarly,
\[
\frac{\partial F_{jl}}{\partial\overline{z}} = \frac{b_{jl}+i\,c_{jl}}{2}\cdot\frac{1}{\overline{z}}
-\frac{v_{jl}}{2}\cdot\frac{1}{\overline{z}} 
=\frac{(b_{jl}+i\,c_{jl})-v_{jl}}{2\overline{z}}.
\]
Thus, for the single–variable factor we have
\[
\frac{\partial}{\partial z} z_l^{a_{jl}} = \frac{(b_{jl}+i\,c_{jl})+v_{jl}}{2z_l}\,z_l^{a_{jl}},
\]
and
\[
\frac{\partial}{\partial\overline{z}} z_l^{a_{jl}} = \frac{(b_{jl}+i\,c_{jl})-v_{jl}}{2\overline{z}_l}\,z_l^{a_{jl}}.
\]

Returning to
\[
u_j(z)= z_j \prod_{l\notin I} z_l^{a_{jl}},
\]
we consider three cases:

\bigskip

\textbf{Case 1.} \(k=j\).\\[1mm]
We have
\[
\frac{\partial u_j}{\partial z_j} = \frac{\partial z_j}{\partial z_j}\,\prod_{l\notin I} z_l^{a_{jl}}
=\prod_{l\notin I} z_l^{a_{jl}} = \frac{u_j(z)}{z_j},
\]
and, since \(z_j\) is holomorphic,
\[
\frac{\partial u_j}{\partial\overline{z}_j} = 0.
\]

\bigskip

\textbf{Case 2.} \(k\neq j\) and \(k\notin I\).\\[1mm]
The only factor that depends on \(z_k\) (and \(\overline{z}_k\)) is 
\[
f_{jk}(z_k) = z_k^{a_{jk}},
\]
with
\[
a_{jk} = (b_{jk}+i\,c_{jk},\,v_{jk}).
\]
Thus, from the formulas above,
\[
\frac{\partial}{\partial z_k}z_k^{a_{jk}} = \frac{(b_{jk}+i\,c_{jk})+v_{jk}}{2z_k}\,z_k^{a_{jk}},
\]
and
\[
\frac{\partial}{\partial\overline{z}_k}z_k^{a_{jk}} = \frac{(b_{jk}+i\,c_{jk})-v_{jk}}{2\overline{z}_k}\,z_k^{a_{jk}}.
\]
Since the other factors in \(u_j(z)\) do not depend on \(z_k\) or \(\overline{z}_k\), the product rule yields
\[
\frac{\partial u_j}{\partial z_k} = z_j \left(\prod_{l\notin I,\;l\neq k} z_l^{a_{jl}}\right)
\frac{\partial}{\partial z_k}z_k^{a_{jk}}
= \frac{(b_{jk}+i\,c_{jk})+v_{jk}}{2z_k}u_j(z),
\]
and
\[
\frac{\partial u_j}{\partial\overline{z}_k} = \frac{(b_{jk}+i\,c_{jk})-v_{jk}}{2\overline{z}_k}u_j(z).
\]

\bigskip

\textbf{Case 3.} \(k\in I\) with \(k\neq j\).\\[1mm]
Here the variable \(z_k\) does not appear in either \(z_j\) or in the product (which is over \(l\notin I\)), so
\[
\frac{\partial u_j}{\partial z_k} = 0,\quad \frac{\partial u_j}{\partial\overline{z}_k} = 0.
\]
Hence
\begin{align*}
&A(z)=\begin{pmatrix}
       diag\Bigl((\frac{u_{j}}{z_{j}})_{j\in I}\Bigr) &\Bigl(\frac{(b_{jk}+i\,c_{jk})+v_{jk}}{2z_k}u_j\Bigr)_{j\notin I}
    \end{pmatrix},& &B(z)=\begin{pmatrix}
        \0 
        &\Bigl(\frac{(b_{jk}+i\,c_{jk})-v_{jk}}{2\overline{z}_k}\, u_j\Bigr)_{j\notin I}
    \end{pmatrix}\\
    &A(x_{0})=\begin{pmatrix}
        I_{n}
        &\Bigl(\frac{(b_{jk}+i\,c_{jk})+v_{jk}}{2}\Bigr)_{j\notin I}
    \end{pmatrix}, & &B(x_{0})=\begin{pmatrix}
    \0 
        &\Bigl(\frac{(b_{jk}+i\,c_{jk})-v_{jk}}{2}\Bigr)_{j\notin I}
    \end{pmatrix}
\end{align*}
 and 
 \[ J_{I}(x_{0})=\begin{pmatrix}
     I_{n} &(b_{jk})_{j\in I, k\notin I}) &\0 &\0\\
     \0 &(c_{jk})_{j\in I, k\notin I}) &I_{n} &(v_{jk})_{j\in I, k\notin I})
 \end{pmatrix}.\]
$f_{x_{0}}:E\to T$ is surjective as $J_{I}(x_{0})$ is of full rank $2n$. Since every map in our construction is $\T$-equivariant, so is $f$, therefore $f$ preserves gradings of $E$ and $T$ and consequently $f(E^{k}(\mu))=T^{k}(\mu)$.\\
$ $\\
{\bf 2. Poset of the kernel:}\\
Let $F=\Ker(f_{x_{0}})$. We first show that $F=\Lie(\Ker\lambda)$.
An element 
\[
h=(h_1,\dots,h_m)\in (\CS)^{m}
\]
lies in $\Ker\lambda$ if and only if
\[
h_i\prod_{k\notin I}h_k^{\langle \alpha^I_i, \beta_k\rangle}=1\quad \text{for every } i\in I
\] by equation (\ref{aibk}).
Now take a smooth curve 
\[
h(t)=(h_1(t),\dots,h_m(t))\quad \text{in } \Ker\lambda,
\]
with \(h(0)=(1,\dots,1)\) and \(h'(0)=(x_{1},\cdots,x_{m},y_{1},\cdots,y_{m})\in \R^{2m}\). Write
\[
h_k(t)=\exp(t(x_{k}+iy_{k})),
\]
then the defining equation becomes, for each \(i\in I\),
\[
\exp\Bigl(t(x_i+iy_i)\Bigr)
\prod_{k\notin I} \exp\Bigl(t\Bigl[(b_{ik}+i\,c_{ik})\,x_k + i\,v_{ik}\,y_k\Bigr]\Bigr)
=\exp\Biggl(t\Bigl[x_i+iy_i+\sum_{k\notin I}\Bigl((b_{ik}+i\,c_{ik})\,x_k + i\,v_{ik}\,y_k\Bigr)\Bigr]\Biggr)=1
\]
for all \(t\). Equivalently,
\[
x_i + iy_i + \sum_{k\notin I}\Bigl[(b_{ik}+i\,c_{ik})\,x_k + i\,v_{ik}\,y_k\Bigr] = 0.
\]
Separating into real and imaginary parts, for every \(i\in I\) we obtain:
\[
\textbf{(Real part):} \quad x_i + \sum_{k\notin I} b_{ik}\,x_k = 0,
\]
\[
\textbf{(Imaginary part):} \quad y_i + \sum_{k\notin I} \Bigl( c_{ik}\,x_k + v_{ik}\,y_k \Bigr) = 0.
\]
This shows exactly that $F=\Lie(\Ker\lambda)$.
$ $\\
Given $I\in \Si^{(n)},k\in I,\mu\in \MR$, consider three cases.\\
Case 1: If $0\geq_{s}\mu$, then 
\[f^{k}_{\mu}=f: E\to T,\]
therefore $F^{k}(\mu)=F$.\\
Case 2: If $1\geq_{s}\mu, 0\ngeq_{s} \mu$, then 
\[f^{k}_{\mu}:E_{k} \simeq T_{\chi^{\alpha^{I}_{k}}} \]
because the corresponding matrix is $1\in GL_{1}(\C)$, therefore $F^{k}(\mu)=0$.\\
Case 3: If $1\ngeq_{s}\mu, 0\ngeq_{s}\mu$, then 
\[f^{k}_{\mu}:0\to 0,\]
therefore $F^{k}(\mu)=0$.\\
Hence $\{F^{k}(\mu)| k\in \Si^{(1)},\mu\in \MR\}$ is exactly the poset associated with $X(\De)\times \Lie(\Ker\lambda)$.
\end{proof}

\subsection{Non-toric example}
In §5 of \cite{IFM12}, the authors constructed the following topological toric manifold $X$ that is not a toric variety. We recall its combinatorial data and compute the poset maps associated with the canonical sequence in theorem ~\ref{canonical} for this manifold. We will use the notation in equation ~\ref{dual}, so a ray $\beta=((b_{1}+ic_{1},v_{1}),(b_{2}+ic_{2},v_{2}))\in \MR^{2}$ is identified with the matrix 
$\begin{pmatrix}
    b_{1} &0\\
    c_{1} &v_{1}\\
    b_{2} &0 \\
    c_{2} &v_{2}
\end{pmatrix}$ and a dual vector $\alpha=((b_{1}+ic_{1},v_{1}),(b_{2}+ic_{2},v_{2}))$ is identified with
$\begin{pmatrix}
    b_{1} &0 &b_{2} &0\\
    c_{1} &v_{1} &c_{2} &v_{2}
\end{pmatrix}$.\\
{\bf Topological fan of $X$:} Let $\Si$ be an abstract simplicial complex defined by \[\Si:=\{\emptyset, \{1\},\{2\},\{3\},\{4\},J_{1}=\{1,2\},J_{2}=\{2,3\},J_{3}=\{3,4\},J_{4}=\{4,1\}\}. \]
The rays are given by 
\begin{align*}
    \beta_{1}&=\begin{pmatrix}
        1 &0\\
        0 &1\\
        0 &0\\
        0 &0
    \end{pmatrix}, 
    &\beta_{2}&=\begin{pmatrix}
        0 &0\\
        0 &0\\
        1 &0\\
        0 &1
    \end{pmatrix},
    &\beta_{3}&=\begin{pmatrix}
        -1 &0\\
        0 &-1\\
        0 &0\\
        0 &-2
    \end{pmatrix},
    &\beta_{4}&=\begin{pmatrix}
        -1 &0\\
        0 &-1\\
        -1 &0\\
        0 &-1
    \end{pmatrix}.
\end{align*}
The dual vectors are given by 
\begin{align*}
    \alpha^{J_{1 }}_{1 }&=\begin{pmatrix}
    1   &0  &0  &0\\
    0  &1   &0  &0
    \end{pmatrix},
    &\alpha^{J_{1 }}_{2 }&=\begin{pmatrix}
    0   &0  &1  &0\\
    0  &0   &0  &1
    \end{pmatrix},\\
    \alpha^{J_{2 }}_{2 }&=\begin{pmatrix}
    0   &0  &1  &0\\
    0  &-2   &0  &1
    \end{pmatrix},
    &\alpha^{J_{2 }}_{3 }&=\begin{pmatrix}
    -1   &0  &0  &0\\
    0  &-1   &0  &0
    \end{pmatrix},\\
    \alpha^{J_{ 3}}_{ 3}&=\begin{pmatrix}
    -1   &0  &1  &0\\
    0  &1   &0  &-1
    \end{pmatrix},
    &\alpha^{J_{3 }}_{4}&=\begin{pmatrix}
    0   &0  &-1  &0\\
    0  &-2   &0  &1
    \end{pmatrix},\\
    \alpha^{J_{ 4}}_{4 }&=\begin{pmatrix}
    0   &0  &-1  &0\\
    0  &0   &0  &-1
    \end{pmatrix},
    &\alpha^{J_{4 }}_{1 }&=\begin{pmatrix}
    1   &0  &-1  &0\\
    0  &1   &0  &-1
    \end{pmatrix}.
\end{align*}
{\bf Poset maps of the canonical sequence:}\\
Step 1: Fix an arbitrary point $x_{0}$ in the open dense torus of $X$. For simplicity, we choose $x_{0}=(1,1)\in \T$. \\
Step 2: Choose an arbitrary top dimensional cone $I$. Here we choose $I=J_{1}$. \\
Step 3: Compute $a_{jl} := \langle \alpha^I_j,\beta_l\rangle = (b_{jl}+i\,c_{jl},\,v_{jl})$. By simple calculation, we have 
\begin{align*}
    a_{13}&=(-1,-1), &a_{14}=(-1,-1)\\
    a_{23}&=(0,-2), &a_{24}=(-1,-1).
\end{align*}
Hence $J_{I}(x_{0})=\begin{pmatrix}
    1 &0 &-1 &-1 &0 &0 &0 &0\\
    0 &1 &0 &-1 &0 &0 &0 &0\\
    0 &0 &0 &0 &1 &0 &-1 &-1\\
    0 &0 &0 &0 &0 &1 &-2 &-1
\end{pmatrix}$ is the matrix form of $f_{x_{0}}$.

\section{Holomorphic Klyachko vector bundles}
\begin{defi}
Let $X(\De)$ be a topological toric manifold with a $\T$-equivariant complex structure. A holomorphic Klyachko vector bundle of rank $r$ over $X(\De)$ is a $\T$-equivariant holomorphic complex vector bundle $\ME\to X(\De)$ whose restriction $\left. \ME \right|_{U_{I}}$ is equivariantly biholomorphic to $U_{I}\times \C^{r}$ for all $I\in \Si$ and whose action maps $\T\times \ME\to \ME$ and $\T\times X(\De)\to X(\De)$ are holomorphic.
\end{defi}

Define $diag(\Z):=\{(v,v)\in \MR | v\in \Z\}$ and $diag(\Z^{n})=\{(\alpha,\cdots,\alpha)\in \MR^{n} | \alpha\in diag(\Z)\}$.

\begin{lemm} \label{holchar}
   If a smooth character $\chi^{\alpha}: \T\to \CS$ is holomorphic for $\alpha\in \MR$, then $\alpha\in diag(\Z)$. 
\end{lemm}

\begin{proof}
Let $\alpha=(b+ci,v)\in \MR$ and $f(\rho,\theta)=\rho^{b+ci}e^{i\theta v}=\rho^{b}e^{ci\ln{\rho}+i\theta v}=U(\rho,\theta)+iV(\rho,\theta)$ where $U(\rho,\theta)=\rho^{b}\cos(c\ln{\rho}+\theta v)$  and $V(\rho,\theta)=\rho^{b}\sin(c\ln{\rho}+\theta v)$. Because $\chi^{\alpha}$ is holomorphic, $U,V$ must satisfy the Cauchy-Riemann equations 
\begin{align*}
    \frac{\partial U}{\partial \rho}&=\frac{1}{\rho}\frac{\partial U}{\partial \theta},\\
    \frac{\partial V}{\partial \rho}&= -\frac{1}{\rho}\frac{\partial U}{\partial \theta}.
\end{align*}
By simple computation, 
\begin{align*}
    \frac{\partial U}{\partial \rho}&= b\rho^{b-1}\cos(w)-\rho^{b-1}c\sin(w),\\
    \frac{1}{\rho}\frac{\partial U}{\partial \theta}&=\rho^{b-1}v\cos(w), \\
    \frac{\partial V}{\partial \rho}&= b\rho^{b-1}\sin(w)+\rho^{b-1}c\cos(w),\\
    -\frac{1}{\rho}\frac{\partial U}{\partial \theta}&= \rho^{b-1}v\sin(w).
\end{align*}
where $w=c\ln{\rho}+\theta v$.\\
After simplification, we get 
\begin{align*}
    (b-v)\cos(w)&=c\sin(w),\\
    (b-v)\sin(w)&=-c\cos(w).
\end{align*}
Therefore $(b-v)^{2}=-c^{2}$, i.e. $b=v\in \Z, c=0$, i.e. $\alpha\in diag(\Z)$.
\end{proof}

\begin{theo} \label{main4}
    A holomorphic Klyachko vector bundle $\ME$ over a complex topological toric manifold $X=X(\De)$ is $\T$-equivariantly biholomorphic to a toric vector bundle over a toric variety. 
\end{theo}

\begin{proof}
In the proof of theorem 1.2 in \cite{Ish11}, Ishida showed that for a topological toric manifold $X$ with $\T$-equivariant complex structure, the equivariant smooth charts $\varphi_{I}:U_{I}\to V(\bigoplus_{i\in I}\chi^{\alpha^{I}_{i}})$ are actually equivariant holomorphic charts. Since the $\T$-action is holomorphic, the map 
\begin{align*}
    \T &\longrightarrow V(\bigoplus_{i\in I}\chi^{\alpha^{I}_{i}})\\
    g &\longmapsto (\chi^{\alpha^{I}_{i}}(g)z_{i})_{i\in I}
\end{align*}
is holomorphic for each $I\in \Si^{(n)}, g\in \T, z\in V(\bigoplus_{i\in I}\chi^{\alpha^{I}_{i}})$, in particular 
\begin{align*}
    \T &\longrightarrow \CS\\
    g &\longmapsto \chi^{\alpha^{I}_{i}}(g)
\end{align*}
is holomorphic for all $I\in \Si^{(n)}, i\in I$. By lemma ~\ref{holchar}, $\alpha^{I}_{i}\in diag(\Z^{n})$. Hence $\beta_{i}\in diag(\Z^{n})$, i.e. the topological toric fan $\De$ comes from an ordinary fan $\Si'$, i.e. $X$ is a toric variety. Now fix a point $x_{0}$ in the open dense torus and let $E$ be the fiber of $\ME$ over $x_{0}$. For each $I\in \Si^{(n)}$, $\left.\ME \right.\!|_{U_{I}}$ is $\T$-equivariantly biholomorphic to $U_{I}\times E$ and the $\T$-module $E$ decomposes into weight spaces 
\[E=\bigoplus_{\chi}E_{\chi}.\]
By lemma ~\ref{holchar}, each $\chi$ appearing in the decomposition must be algebraic. In the proof of theorem ~\ref{main}, we showed that the transition functions $f_{IJ}$ of $\ME$ is determined by the weighs $\chi$, hence $f_{IJ}$ are algebraic and $\ME$ is exactly a toric vector bundle.
\end{proof}

\end{document}